\RequirePackage{fix-cm}
\documentclass[smallextended]{svjour3}       
\smartqed  

\usepackage{graphicx}
\usepackage{amsmath, ntheorem}
\usepackage{amssymb}
\usepackage{stackrel}
\usepackage{marvosym, enumerate}
\usepackage{amstext, amscd, latexsym}
\usepackage{amsfonts, ragged2e}
\usepackage{mathrsfs, pgfplots, enumerate, setspace}
\usepackage{subfigure}
\usepackage{cases}
\smartqed
\usepackage{cite}
\usepackage{amstext, graphicx, amscd, latexsym, amssymb}
\usepackage{amsfonts, ragged2e}
\usepackage{pgfplots, setspace}
\usepackage{geometry}
\usepackage{bm}
\usepackage{float}
\usepackage{epstopdf,caption}
\usepackage{mathtools}
\usepackage{amsfonts,amsmath,amsbsy,amssymb,amscd,mathrsfs}
\usepackage{graphicx}
\usepackage{epstopdf}
\usepackage{algorithmic} 
\usepackage{graphicx,epstopdf}

 \newcommand{\nm}[1]{\left\lVert {#1} \right\rVert}

\usepackage[figuresright]{rotating}
\usepackage[sort&compress,numbers]{natbib}
\usepackage[ruled,linesnumbered]{algorithm2e} 
\usepackage[colorlinks,linkcolor=blue,anchorcolor=blue,citecolor=blue,urlcolor=blue]{hyperref} 

\newtheorem{assumption}{Assumption}

\journalname{}

\begin{document}

\title{Preconditioned Proximal Gradient Methods with Conjugate Momentum: A Subspace Perspective}


\author{Jian Chen \and  Xinmin Yang}

\institute{J. Chen \at College of Mathematics, Sichuan University, Chengdu 610065, China\\
                    \href{mailto:chenjian_math@163.com}{chenjian\_math@163.com} \\
        \Letter X.M. Yang \at National Center for Applied Mathematics in Chongqing,  Chongqing Normal University, Chongqing 401331, China\\
        \href{mailto:xmyang@cqnu.edu.cn}{xmyang@cqnu.edu.cn}  \\}
\date{Received: date / Accepted: date}

\maketitle

\begin{abstract}
In this paper, we propose a descent method for composite optimization problems with linear operators. 
Specifically, we first design a structure-exploiting preconditioner tailored to the linear operator so that the resulting preconditioned proximal subproblem admits a closed-form solution through its dual formulation. 
However, such a structure-driven preconditioner may be poorly aligned with the local curvature of the smooth component, which can lead to slow practical convergence.
To address this issue, we develop a subspace proximal Newton framework that incorporates curvature information within a low-dimensional subspace. 
At each iteration, the search direction is obtained by minimizing a proximal Newton model restricted to a two-dimensional subspace spanned by the current preconditioned proximal gradient direction and a momentum direction derived from the previous iterate. 
By orthogonalizing the subspace basis with respect to the local Hessian-induced metric, the solution of the original coupled two-dimensional nonsmooth subproblem can be well approximated by solving two decoupled one-dimensional subproblems, while keeping the per-iteration computational cost low.
We establish global convergence of the proposed method and prove a $Q$-linear convergence rate under strong convexity. 
Comparative numerical experiments demonstrate the effectiveness of the proposed algorithm on ill-conditioned problems.

\end{abstract}

\section{Introduction}
\label{sec:intro}
Composite optimization problems arise in a wide range of applications, including machine learning, signal processing, and data science. A typical formulation is
\begin{equation*}
	\min_{x \in \mathbb{R}^n} F(x) = f(x) + g(x),
\end{equation*}
where $f: \mathbb{R}^{n} \rightarrow \mathbb{R}$ is a continuously differentiable function, $g: \mathbb{R}^{n} \rightarrow (-\infty,+\infty]$ is a proper, convex, and lower semicontinuous function but not necessarily differentiable. Due to their broad applicability, the design of efficient algorithms for such problems has attracted considerable attention.

In recent years, first-order methods have attracted much attention because of their low computational cost and scalability to large-scale problems. Representative approaches include the proximal gradient method \cite{N2013}, the iterative shrinkage-thresholding algorithm (ISTA) \cite{BT2009}, and various primal-dual methods \cite{CP2011,HY2012,MP2018}. These methods only require gradient information and proximal operators, making them attractive for high-dimensional applications. However, it is well known that the performance of first-order methods can deteriorate significantly when the problem is ill-conditioned, which is frequently encountered in practical applications.

To address this issue, second-order methods have been extensively studied. Classical proximal Newton-type methods \cite{LSS2014,ST2016} exploit curvature information to achieve faster local convergence. Nevertheless, the main difficulty of second-order approaches lies in the computational cost of evaluating or approximating Hessian matrices, which becomes prohibitive for large-scale problems. Besides, introducing non-diagonal preconditioning matrices typically destroys the closed-form property of proximal operators. As a consequence, many practical algorithms rely on diagonal preconditioning \cite{PDB2020,QGHY2025} to preserve computational efficiency. Although such approaches often yield improved practical performance, establishing theoretical guarantees comparable to those of vanilla proximal Newton and quasi-Newton methods is generally challenging.

Between first- and second-order methods lies another important class of optimization techniques: momentum methods, which exploit historical information to accelerate convergence. Classical frameworks of momentum methods include the heavy-ball method \cite{P1964}, Nesterov's accelerated gradient method \cite{N1983}, and the nonlinear conjugate gradient methods \cite{DY1999,FR1964,HS1952,PR1969}. Due to their low computational cost and strong empirical performance, momentum-based methods have become widely used in large-scale optimization problems.

Understanding why momentum methods perform well has attracted considerable attention in recent years. For instance, the heavy-ball method and Nesterov's accelerated gradient method have been studied and interpreted from the perspective of inertial dynamical systems \cite{ACFR2022,LC2022,SBC2016,SDSJ2022}. The nonlinear conjugate gradient method, on the other hand, originates from the conjugacy property of the linear conjugate gradient method and its finite termination property on quadratic problems. However, in large-scale or ill-conditioned settings, the conjugacy property may deteriorate, leading to degraded performance.

To mitigate this issue, Yuan and Stoer \cite{YS1995} proposed a subspace algorithm. The main idea is to construct a low-dimensional subspace using historical search directions together with the current gradient direction. Within this subspace, a second-order model is approximated using finite differences, and the optimal solution of the subspace problem is used to determine conjugate parameters. This strategy can be interpreted as a subspace optimality principle. When the basis and the approximation model are properly chosen, the resulting method can achieve fast asymptotic convergence comparable to full-space second-order algorithms \cite{ZGH2022}. As a result, subspace algorithms have attracted increasing attention. More recently, Lapucci et al. \cite{LLLS2026} established global convergence results for subspace methods in the nonconvex setting. This strategy has been extended to multi-objective \cite{CTY2025} and constrained optimization problems \cite{LLL2026}.

Given the strong empirical and theoretical performance of subspace algorithms in smooth optimization, a natural question arises:
\begin{center}\it
Can the subspace algorithm of Yuan and Stoer be extended to composite optimization problems?
\end{center}

In this paper, we address this question by studying a more general class of composite optimization problems with linear operators.
\begin{equation} \label{eq:main}
	\min_{x \in \mathbb{R}^n} F(x) = f(x) + g(Ax),
\end{equation}
where $f: \mathbb{R}^{n} \rightarrow \mathbb{R}$ is a continuously differentiable function, $g: \mathbb{R}^{m} \rightarrow (-\infty,+\infty]$ is a proper, convex, and lower semicontinuous function but not necessarily differentiable, and $A\in\mathbb{R}^{m\times n}$. Models of the form \eqref{eq:main} arise in a wide range of applications such as imaging
(especially total variation regularization), sparse recovery, and various linearly constrained
optimization problems. Extending subspace algorithms to problems of the form \eqref{eq:main} presents several fundamental challenges. First, the proximal operator of the composition $g \circ A$ does not admit a closed-form solution, which makes it difficult to construct efficient descent directions and the associated subspace. In particular, the choice of search directions is crucial for capturing useful curvature information and maintaining desirable convergence properties. Second, even when the search is restricted to a low-dimensional subspace, the resulting nonsmooth subproblem may still be difficult to solve efficiently. Designing computational strategies that balance approximation accuracy and computational efficiency therefore becomes a central challenge in extending subspace methods to composite optimization.

To address these challenges, this paper explores efficient strategies for subspace construction and fast algorithms for solving the subspace model. The main contributions are summarized as follows.
\begin{itemize}
	\item[$\bullet$] When the linear operator $A$ has full row rank, we exploit the singular value decomposition (SVD) of $A$ and complete the orthogonal eigenvectors of $A^{\top}A$ to construct a preconditioner $P$, which enables the preconditioned dual problem to admit a closed-form solution. In the linear case where $A$ is not of full row rank, we remove inactive constraints so that the remaining rows become full rank, allowing the same preconditioning strategy to be applied. We further investigate the applications of this preconditioning strategy, particularly in linearly constrained optimization problems. By exploiting the structure induced by the proposed preconditioner, we construct a closed-form preconditioned projection, which significantly simplifies the projection step that typically arises in projected gradient methods. As a result, the computational difficulty of evaluating projections onto the feasible set is greatly reduced \cite{C2016}. 
	\item[$\bullet$] The structure-driven preconditioner $P$ may severely misalign with the local curvature of the smooth
	term \(f\), leading to an ill-conditioned transformed Hessian $P^{-1/2}\nabla^{2}f(x)P^{-1/2}$ in the transformed problem. To alleviate this issue, motivated by the affine-invariant property of Newton's method, we introduce a subspace second-order model that improves the conditioning of the problem. Specifically, a local second-order approximation is constructed within a low-dimensional subspace using finite differences. To balance approximation accuracy and computational efficiency, we identify two conjugate directions associated with the local curvature within this subspace. Along each direction, a one-dimensional second-order model is constructed and solved efficiently. The resulting search direction consists of a preconditioned descent direction and its conjugate direction, which leads to a preconditioned proximal gradient method with conjugate momentum ($\mathrm{P^{2}GM}$\_CM).
	\item[$\bullet$]Numerical experiments demonstrate that the conjugate direction plays a crucial role in the performance of the proposed algorithm. The resulting method achieves competitive performance on several problems, including LASSO problems, linearly constrained optimization problems, and structured $\ell_{1}$ regularization problems.
	
\end{itemize}
 
\par The rest of this paper is organized as follows. 
Section \ref{sec2} introduces the notation and preliminary results that will be used throughout the paper. 
In Section \ref{sec3} we develop the preconditioned proximal gradient framework and describe the construction of the structure-exploiting preconditioner. 
Section \ref{sec4} discusses several important applications in which the proposed preconditioning strategy leads to closed-form dual updates. 
In Section \ref{sec5} we introduce the subspace acceleration mechanism and the conjugate momentum strategy, and derive the associated search directions. 
Section \ref{sec6} presents the complete algorithm together with efficient methods for solving the one-dimensional subproblems and establishes the global and linear convergence results. 
Numerical experiments demonstrating the efficiency of the proposed method are reported in Section \ref{sec7}.  Finally, some conclusions are drawn at the end of the paper.

\section{Preliminaries}\label{sec2}
Throughout this paper, the $n$-dimensional Euclidean space $\mathbb{R}^{n}$ is equipped with the inner product $\langle\cdot,\cdot\rangle$ and the induced norm $\|\cdot\|$. 
Denote by $\mathbb{S}^{n}_{++}$ ($\mathbb{S}^{n}_{+}$) the set of symmetric positive (semi-)definite matrices and by $\mathbb{O}^{n}$ the set of orthogonal matrices in $\mathbb{R}^{n\times n}$. 
The rank of a matrix is denoted by $\mathcal{R}(\cdot)$. 
For a differentiable function $f$, $\nabla f(x)\in\mathbb{R}^{n}$ and $\nabla^{2}f(x)\in\mathbb{R}^{n\times n}$ denote the gradient and the Hessian of $f$ at $x$, respectively. 
For a positive definite matrix $H$, we define the norm
\[
\|x\|_{H}=\sqrt{\langle x,Hx\rangle}.
\]
For simplicity, we denote $[n]:=\{1,2,\ldots,n\}$ and define the $n$-dimensional unit simplex by
\[
\Delta_{n}:=\left\{x\in\mathbb{R}^n:\sum_{i\in[n]}x_i=1,\ x_i\ge0\right\}.
\]
To avoid ambiguity, we introduce the partial order $\preceq(\prec)$ in $\mathbb{R}^n$ as
\[
u\preceq(\prec)v \iff v-u\in\mathbb{R}^n_{+}(\mathbb{R}^n_{++}),
\]
and in $\mathbb{S}^n$ as
\[
U\preceq(\prec)V \iff V-U\in\mathbb{S}^n_{+}(\mathbb{S}^n_{++}).
\]
For $b\in\mathbb{R}^m$, we denote the interval
\[
[-\infty,b]=[-\infty,b_1]\times\cdots\times[-\infty,b_m].
\]
For $a\in\mathbb{R}^n$, $r>0$, and $p\in[1,+\infty]$, the $\ell_p$ ball centered at $a$ with radius $r$ is defined as
\[
\mathrm{B}_p[a,r]:=\{x\in\mathbb{R}^n:\|x-a\|_p\le r\}.
\]
For a vector $x\in\mathbb{R}^{n}$ and an interval $[u,v]$ satisfying $v-u\in\mathbb{R}^n_{+}$, we define the componentwise clipping operator by
$$\mathtt{clip}\left(x,u,v\right):=(\min\{v_{i},\max\{u_{i},x_{i}\}\})_{i=1}^{n}.$$

For a proper extended real-valued function $h:\mathbb{R}^n\to(-\infty,+\infty]$, its domain is defined as
\[
\mathrm{dom}\,h:=\{x\in\mathbb{R}^n:h(x)<+\infty\}.
\]
The subdifferential of $h$ at $x\in\mathrm{dom}\,h$ is defined by
\[
\partial h(x):=\{v\in\mathbb{R}^n: h(y)\ge h(x)+\langle v,y-x\rangle,\ \forall y\in\mathbb{R}^n\}.
\]
The convex conjugate of $h$ is defined as
\[
h^*(y)=\sup_{x\in\mathbb{R}^n}\{\langle x,y\rangle-h(x)\}.
\]
The proximal operator associated with $h$ is defined by
\[
\mathrm{prox}_h(x)
=\arg\min_{u\in\mathbb{R}^n}
\left\{
h(u)+\frac12\|u-x\|^2
\right\}.
\]

Let $C\subseteq\mathbb{R}^n$ be a nonempty closed convex set. The normal cone of $C$ at $x\in C$ is defined by
\[
N_C(x):=\{v\in\mathbb{R}^n:\langle v,y-x\rangle\le0,\ \forall y\in C\}.
\]
The support function of $C$ is defined as
\[
\sigma_C(x):=\sup_{y\in C}\langle x,y\rangle,\quad x\in\mathbb{R}^n.
\]
The indicator function of $C$ is defined by
\[
\delta_C(x)=
\begin{cases}
	0, & x\in C,\\
	+\infty, & x\notin C.
\end{cases}
\]
For a symmetric positive definite matrix $H\in\mathbb{S}^n_{++}$ and a nonempty closed convex set $C\subseteq\mathbb{R}^n$, the preconditioned projection of $x$ onto $C$ is defined as
\[
\Pi_C^H(x):=\arg\min_{y\in C}\|y-x\|_H.
\]
It holds that $z=\Pi_C^H(x)$ if and only if
\[
H(x-z)\in N_C(z).
\]

A differentiable function $f$ is said to be $L$-smooth if
\[
\|\nabla f(x)-\nabla f(y)\|
\le L\|x-y\|,
\quad \forall x,y\in\mathbb{R}^n.
\]
This implies the descent inequality
\[
f(y)\le f(x)+\nabla f(x)^\top(y-x)
+\frac{L}{2}\|y-x\|^2.
\]
The function $f$ is $\mu$-strongly convex if
\[
f(y)\ge f(x)+\nabla f(x)^\top(y-x)
+\frac{\mu}{2}\|y-x\|^2
\]
for all $x,y\in\mathbb{R}^n$.

\section{Preconditioned proximal gradient method}\label{sec3}
In this section we develop a preconditioned proximal gradient framework for solving composite optimization problems of the form (\ref{eq:main}). In general, the proximal mapping of $g\circ A$ does not admit a closed-form expression, which may significantly increase the cost of each proximal gradient step. To address this difficulty, we introduce a preconditioning strategy that transforms the proximal gradient subproblem into a dual problem with a much simpler structure. 

For $P\in\mathbb{S}^{n}_{++}$,  consider the following preconditioned proximal gradient subproblem
\begin{equation}\label{prime}
	\min_{x\in\mathbb{R}^{n}}
	\nabla f(x^{k})^{\top}(x-x^{k})
	+
	g(Ax)
	+
	\frac{1}{2}\|x-x^{k}\|^{2}_{P}.
\end{equation}
To analyze this problem, we introduce its saddle-point formulation
$$\min_{x\in\mathbb{R}^{n}}\max_{y\in\mathbb{R}^{m}}\nabla f(x^{k})^{\top}(x-x^{k})+\frac{1}{2}\|x-x^{k}\|^{2}_{P}+y^{\top} Ax-g^{*}(y),$$
where $g^*$ is the convex conjugate of $g$. By the minimax theorem, the problem can be equivalently written as
\begin{align*}
\max_{y\in\mathbb{R}^{m}}\min_{x\in\mathbb{R}^{n}}\nabla f(x^{k})^{\top}(x-x^{k})+\frac{1}{2}\|x-x^{k}\|^{2}_{P}+y^{\top}Ax-g^{*}(y).
\end{align*}
Let $\tilde{x}^{k}$ denote the minimizer of problem \eqref{prime}. 
The optimality condition with respect to $x$ yields
\begin{equation}\label{x}
	\tilde{x}^{k}=x^{k}-P^{-1}(\nabla f(x^{k})+A^{\top}y^{k}),
\end{equation}
where $y^{k}$ is the optimal solution of the following dual problem:
\begin{equation}\label{dual}
	\min_{y\in\mathbb{R}^{m}}\frac{1}{2}y^{\top}AP^{-1}A^{\top}y+g^*(y) -{a^{k\top}}y,
\end{equation}
with $$a^{k}:=Ax^{k}-AP^{-1}\nabla f(x^{k}).$$
\par The main motivation behind the preconditioned proximal gradient method is to simplify the dual subproblem through a proper choice of the preconditioner $P$. In particular, if $P$ is chosen such that 
\begin{equation}\label{inv}
	AP^{-1}A^{\top}=\bm I_{m},
\end{equation}
then the dual problem \eqref{dual} reduces to
\[
\min_{y\in\mathbb{R}^{m}}
\frac12\|y\|^2 + g^*(y) - a^{k\top}y,
\]
whose solution is simply
$$y^{k}=\mathrm{prox}_{g^*}(a^{k}).$$
Therefore, the key question becomes how to construct a suitable preconditioner $P$ that satisfies condition \eqref{inv} while remaining computationally tractable.
\subsection{Selection of $P$}
We now describe a systematic way to construct a preconditioner satisfying condition \eqref{inv}. The construction relies on the SVD of the matrix $A$.

Let the SVD of $A$ be
$$A = U\Lambda V^{\top},$$
where $U\in\mathbb{O}^{m}$ and $V\in\mathbb{O}^{n}$. Then
 $$A^{\top}A=V\Lambda^{\top}\Lambda V^{\top}.$$
Based on this decomposition, we select the preconditioner $P$ as
\begin{equation}\label{p}
	P=\left\{
	\begin{aligned}
		&A^{\top}A, & \mathcal{R}(A)&=m= n, \\
		&A^{\top}A+ V\begin{bmatrix}
			\bm0_{m\times m}   & &  \\
			
			&  &\tilde{P}
		\end{bmatrix}V^{\top}, &  \mathcal{R}(A)&=m<n,
	\end{aligned}
	\right.
\end{equation}
where $\tilde{P}\in\mathbb{S}^{n-m}_{++}$. 
To further reveal the structure of the preconditioner, define
\begin{equation}\label{m}
	M=\left\{
	\begin{aligned}
		&A, & \mathcal{R}(A)&=m= n, \\
		&\begin{bmatrix}
			&A  \\
			&\bm0_{(n-m)\times n}
		\end{bmatrix}+ \begin{bmatrix}
			\bm0_{m\times m}   & &  \\
			
			&  &\tilde{P}^{1/2}
		\end{bmatrix}V^{\top}, &  \mathcal{R}(A)&=m<n.
	\end{aligned}
	\right.
\end{equation}
With this construction, the preconditioner can be written as
 $$P = M^{\top}M.$$
Consequently, $P \in \mathbb{S}^{n}_{++}$ and satisfies condition \eqref{inv}. 
This choice of $P$ ensures that the dual subproblem admits a closed-form solution, which significantly simplifies the computation of the proximal gradient step.
\subsection{Diagonal preconditioning}
A main appeal of iteration \eqref{prime} is that, for suitable choices of $P$, the subproblem \eqref{dual} admits a closed-form solution, making each iteration computationally efficient. It is worth noting, however, that such a choice of $P$ is tightly coupled with the linear operator $A$ in order to guarantee $AP^{-1}A^{\top}=I_{m}$. While this property greatly simplifies the dual update, it may introduce a potential drawback.

Specifically, since $P$ is primarily designed to accommodate the structure of $A$, it may be poorly aligned with the local second-order geometry of $f$. When the spectral structure of $A$ differs substantially from that of the local curvature $\nabla^{2} f(x^{k})$, the induced metric $\|\cdot\|_{P}$ may impose an inappropriate scaling across different directions. This mismatch may lead to additional ill-conditioning and thus slow down the convergence of the overall algorithm, even though the subproblem itself remains easy to solve.

To strike a balance between per-iteration computational cost and improved curvature exploration, we propose a diagonal preconditioning strategy. Instead of fixing $P$, we allow a variable preconditioner $P=P_k$ in \eqref{dual} such that the quadratic term in the dual subproblem becomes diagonal; equivalently, we enforce that $AP_k^{-1}A^{\top}$ is a diagonal full-rank matrix. This diagonal structure decouples the dual variables and allows direction-wise scaling to be adjusted independently. As a result, the preconditioner $P_k$ can better capture the local curvature information while preserving the computational simplicity of the update. The remaining question is how to construct such a preconditioner $P_k$.

To obtain a diagonal but not necessarily identity matrix in the dual
quadratic term, we first construct an $A$-adapted change of variables.
Assume that $A$ has full row rank, i.e., $\operatorname{rank}(A)=m\le n$,
and let
\[
A=U[\Sigma\ \ 0]V^\top
\]
be its singular value decomposition, where $\Sigma\in\mathbb R^{m\times m}$
is nonsingular. The matrix $M$ defined in (\ref{m}) can be written as
\[
M=
\begin{bmatrix}
	U\Sigma & 0\\
	0 & \widetilde P^{1/2}
\end{bmatrix}
V^\top .
\]
Then
\[
AM^{-1}=[I_m\ \ 0].
\]

Let
\[
\Gamma_k=\operatorname{diag}(\gamma_1^k,\ldots,\gamma_m^k)
\in\mathbb S^m_{++}
\]
be a diagonal scaling matrix, and let
\[
E_k\in\mathbb S^{n-m}_{++}
\]
be diagonal. We define
\[
D_k=
\begin{bmatrix}
	\Gamma_k^{-1} & 0\\
	0 & E_k
\end{bmatrix},
\qquad
P_k=M^\top D_k M .
\]
Then
\[
AP_k^{-1}A^\top
=
[I_m\ \ 0]D_k^{-1}
\begin{bmatrix}
	I_m\\
	0
\end{bmatrix}
=
\Gamma_k .
\]
Thus the dual quadratic term is diagonal with a prescribed positive
diagonal matrix $\Gamma_k$, rather than the identity matrix.

From a preconditioning perspective, the matrix $P=M^{\top}M$ in \eqref{prime} can be interpreted as inducing a change of variables $y=Mx$. Denote
\[
h(y) = f(x) = f(M^{-1}y).
\]
We then apply a diagonal preconditioning matrix $D_k$ to the function $h$ in order to better capture its local geometry. The diagonal matrix $D_k$ can be constructed using various diagonal preconditioning strategies, such as the diagonal Barzilai--Borwein method \cite{PDB2020}. 

\section{Application}\label{sec4}
In this section we illustrate how the proposed preconditioning strategy can significantly simplify the computation of the dual subproblem for several important classes of composite optimization problems. 
\subsection{Ellipsoidal constrained problems}
Consider the ellipsoidal constrained optimization problem:
\begin{align*}
	\min&~~~f(x)         \\
	{\rm s.t.}&~~~x^{\top}Bx\leq b,
\end{align*}
where $B\in\mathbb{S}^{n}_{+}$ and $\mathcal{R}(B)=m$. Let $A$ be a matrix satisfying $A^{\top}A = B$. Then the constraint can be written in the form
\[
g(Ax) = \delta_{\mathrm{B}_2[\bm 0_m,\sqrt{b}]}(Ax),
\]
where $\mathrm{B}_2[\bm 0_m,\sqrt{b}]$ denotes the Euclidean ball with radius $\sqrt{b}$. The conjugate function of $g$ is
\begin{equation}
	g^*(y)=\sigma_{\mathrm{B}_{2}[\bm0_{m},\sqrt{b}]}(y)=\max_{c\in\mathrm{B}_{2}[\bm0_{m},\sqrt{b}]}c^{\top}y.
\end{equation}
Applying the preconditioner $P_{k}=M^{\top}D_{k}M$, the resulting dual problem becomes
$$
\min_{y\in \mathbb{R}^{m}}\max_{c\in\mathrm{B}_{2}[\bm0_{m},\sqrt{b}]}
\frac{1}{2}\|y\|_{\Gamma_k}^{2}+c^{\top}y-(a^{k})^{\top}y.
$$
By minimax theorem, it is equivalent to
$$
\max_{c\in\mathrm{B}_{2}[\bm0_{m},\sqrt{b}]}\min_{y\in \mathbb{R}^{m}}
\frac{1}{2}\|y\|_{\Gamma_k}^{2}+c^{\top}y-(a^{k})^{\top}y.
$$
Therefore, by the optimality condition of the minimization problem, we obtain
$$y^{k}=\Gamma_k^{-1}\left(a^{k}-c^*\right),$$
where $c^*$ is the unique optimal solution of the following dual problem:
\begin{align*}
	-\min_{c\in\mathrm{B}_{2}[\bm0_{m},\sqrt{b}]}\frac{1}{2}\nm{a^{k}-c}^{2}_{\Gamma_k^{-1}}.
\end{align*}
The associated Lagrangian is given by
$$
\mathcal{L}(c,\xi)
=
\frac{1}{2}\nm{a^{k}-c}^{2}_{\Gamma_k^{-1}}
+\xi\left(\frac{1}{2}\|c\|^{2}-\frac{b}{2}\right),
$$
where $\xi\geq 0$. By the KKT conditions, the solution $c^*$ satisfies
$$
\xi_k\Gamma_k c^*+ c^*=a^{k},
$$
together with the complementarity condition
$$
\xi_k\left(\frac{1}{2}\|c^*\|^{2}-\frac{b}{2}\right)=0.
$$

If the unconstrained minimizer is strictly feasible, namely,
$$
\|a^{k}\|< \sqrt{b},
$$
then $\xi_k=0$ and hence
$$
c^*=a^{k}.
$$
Otherwise, the ball constraint is active, that is,
$$
\|c^*\|=\sqrt{b},
$$
and
$$
c^*=(\xi_k \Gamma_k+I_m)^{-1}a^{k}
$$
for some $\xi_k>0$. Since
$$
\Gamma_k=\operatorname{diag}(\gamma_1^k,\ldots,\gamma_m^k)
\in\mathbb S^m_{++},
$$
the multiplier $\xi_k$ is determined by the scalar nonlinear equation
$$
\sum_{i=1}^{m}
\left(
\frac{a_i^{k}}{\xi_k\gamma_i^{k}+1}
\right)^2
=
b.
$$
This equation is monotone in $\xi_k$ and can be efficiently solved by Newton's method.

In particular, if $\gamma_1^k=\cdots=\gamma_m^k=\gamma^k$, then the active-case solution reduces to
$$
c^*=
\sqrt{b}\frac{a^{k}}{\|a^{k}\|}.
$$
Therefore, in this special case,
$$
c^*
=
\begin{cases}
	{a^{k}}, 
	& \text{if } \|a^{k}\|< \sqrt{b}, \\[2ex]
	\sqrt{b}\dfrac{a^{k}}{\|a^{k}\|}, 
	& \text{otherwise}.
\end{cases}
$$

\subsection{Structured $\ell_{1}$ regularization problems}
Next we consider a structured $\ell_1$ regularization problem:
\begin{align*}
	\min\limits_{x\in\mathbb{R}^{n}}f(x)  +\lambda \|Ax\|_{1},      
\end{align*}
where $\lambda>0$, $A\in\mathbb{R}^{m\times n}$ and $\mathcal{R}(A)=m$. In this case
\[
g(Ax) = \lambda \|Ax\|_1.
\]
The conjugate function is
\begin{equation}
	g^*(y)=\delta_{\mathrm{B}_{\infty}[\bm0_{m},\lambda]}(y)=\left\{
	\begin{aligned}
		&0, & y\in \mathrm{B}_{\infty}[\bm0_{m},\lambda], \\
		&+\infty, &  \textrm{otherwise}.
	\end{aligned}
	\right.
\end{equation}
Using the preconditioner $P_k = M^\top D_k M$, the dual problem becomes
$$	\min_{y\in \mathrm{B}_{\infty}[\bm0_{m},\lambda]}\frac{1}{2}\nm{y}_{\Gamma_k}^{2} -a^{k\top}y.$$
The optimality condition leads to the explicit solution
\begin{align*}
	y^{k}=\mathrm{\Pi}_{\mathrm{B}_{\infty}[\bm0_{m},\lambda]}(\Gamma_k^{-1}a^{k}).
\end{align*}
Hence the dual update reduces to a projection onto an $\ell_\infty$ ball, which is equivalent to a simple componentwise clipping operation.
\subsection{Linear constrained optimization problems}
When the linear operator $A$ is not of full row rank, the preconditioning matrix can be constructed using the linearly independent components of $A$. This observation is particularly useful for linear constrained optimization problems, where redundant or inactive constraints can be removed so that the remaining constraint matrix becomes full row rank. Let us consider the linear constrained optimization problem:
\begin{align*}
	&\min f(x)         \\
	&{\rm s.t.}~c_{l}\preceq Bx\preceq c_{u},\\
	 &~~~~~Cx=c_{e},
\end{align*}
where $B\in\mathbb{R}^{p\times n}$ and $C\in\mathbb{R}^{q\times n}$. At iteration $k$, we define the working constraint matrix
\[
A_k =
\begin{bmatrix}
	B_k \\
	C
\end{bmatrix},
\]
where $B_k$ consists of the selected inequality constraints.  The corresponding nonsmooth term is
\[
g_k(A_k x) = \delta_{[c_{l},c_u] \times \{c_{e}\}}(A_k x).
\]
The conjugate function of $g_k$ is given by \begin{equation}\label{conju}
	g_{k}^*(y)=\sigma_{[c_{l},c_u]\times \{c_{e}\} }(y)=\max_{c\in[c_{l},c_u]} c^{\top}y_{[1:p_{k}]}+c_{e}^{\top}y_{[p_{k}+1,p_{k}+q]}
\end{equation}
Using the preconditioner $P_{k}=M_{k}^{\top}D_{k}M_{k}$, the dual problem becomes:
$$	\min_{y\in\mathbb{R}^{p_{k}+q}}\max_{c\in[c_{l},c_u]}\frac{1}{2}\nm{y}_{\Gamma_k}^{2}+c^{\top}y_{[1:p_{k}]}+c_{e}^{\top}y_{[p_{k}+1,p_{k}+q]} -a^{k\top}y.$$
By minimax theorem, it is equivalent to
$$	\max_{c\in[c_{l},c_u]}\min_{y\in\mathbb{R}^{p_{k}+q}}\frac{1}{2}\nm{y}_{\Gamma_k}^{2}+c^{\top}y_{[1:p_{k}]}+c_{e}^{\top}y_{[p_{k}+1,p_{k}+q]} -a^{k\top}y.$$
Therefore, by the optimality condition of the minimization problem, we obtain
$$y^{k}=\Gamma_k^{-1}\left(\left\{a^{k}_{[1,p_{k}]}-c^*\right\}\times\left\{a^{k}_{[p_{k}+1,p_{k}+q]}-c_{e}\right\}\right),$$
where $c^*$ is the unique optimal solution of the following dual problem:
\begin{align*}
	-\min_{c\in[c_{l},c_u]}\frac{1}{2}\nm{a^{k}_{[1,p_{k}]}-c}^{2}_{\Gamma_k^{-1}}.
\end{align*}
Since $\Gamma_k$ is diagonal, it follows that
$$y^{k}=\Gamma_k^{-1}\left(\left\{a^{k}_{[1,p_{k}]}-\mathtt{clip}\left(a^{k}_{[1,p_{k}]},c_{l},c_{u}\right)\right\}\times\left\{a^{k}_{[p_{k}+1,p_{k}+q]}-c_{e}\right\}\right).$$
The key ingredient for applying the proposed framework to linear constrained optimization problems lies in identifying the working matrix $A_k$. In contrast to classical active-set methods, which attempt to precisely identify the active constraints, our approach only requires excluding several clearly inactive constraints. In many practical situations, removing inactive indices is significantly easier than accurately detecting the full active set. 

\subsubsection{Simplex constrained optimization problems}
Consider the simplex constrained optimization problem:
\begin{align*}
	\min&~~~f(x)         \\
	{\rm s.t.}&~~~\bm0_{n}\preceq x,\\
	&~~~\bm 1_{n}^{\top}x=1.
\end{align*}
In this case, $A_{k}$ is constructed from the identity matrix by replacing the row corresponding to the active dual index $i_k$ with the vector $\bm 1_n^{\top}$. More precisely,
$$A_{k}=\bm I_{n} +e_{i_{k}}(\bm1_{n}-e_{i_{k}})^{\top},$$
and 
$$g_{k}(A_{k}x)=\delta_{[\bm0_{i_{k}-1},+\infty]\times \{1\} \times [\bm0_{n-i_{k}},+\infty]}(A_{k}x).$$
The conjugate function is
\begin{equation}
	g_{k}^*(y)=\sigma_{[\bm0_{i_{k}-1},+\infty]\times \{1\} \times [\bm0_{n-i_{k}},+\infty]}(y).
\end{equation}
Since $A_{k}$ is a rank-one update of the identity matrix, the Sherman--Morrison formula yields the closed-form expression 
\[
A_{k}^{-1}
= \left(\bm I_n + e_{i_{k}}(\bm1_{n}-e_{i_{k}})^{\top}\right)^{-1}
= \bm I_n - \frac{e_{i_{k}}(\bm1_{n}-e_{i_{k}})^{\top}}{1+(\bm1_{n}-e_{i_{k}})^{\top}e_{i_{k}}}=\bm I_n - e_{i_{k}}(\bm1_{n}-e_{i_{k}})^{\top}.
\]
Substituting $c_{l}=\bm 0_{n-1}$, $c_{u}=\{+\infty\}^{n-1}$ and $c_{e}=1$ into (\ref{conju}) and
using the preconditioner $P_{k}=A_{k}^{\top}D_{k}A_{k}$, the optimality condition gives
$$y^{k}=\Gamma_k^{-1}\left(a^{k}_{[1:i_{k}-1]}-\left[a^{k}_{[1:i_{k}-1]}\right]_{+}\times\{a^{k}_{i_{k}}-1\}\times\left\{a^{k}_{[i_{k}+1:n]}-\left[a^{k}_{[i_{k}+1:n]}\right]_{+}\right\}\right).$$

\subsubsection{Capped simplex constrained optimization problems}
Next, consider the capped simplex constrained optimization problem with parameter $s>0$:
\begin{align*}
	\min&~~~f(x)         \\
	{\rm s.t.}&~~~\bm0_{n}\preceq x\preceq\bm1_{n},\\
	&~~~\bm 1_{n}^{\top}x\leq s.
\end{align*}
In this case, $A_{k}$ is constructed from the identity matrix by replacing the row corresponding to the active dual index $i_k$ with the vector $\bm 1_n^{\top}$. More precisely, 
$$A_{k}=\bm I_{n} +e_{i_{k}}(\bm1_{n}-e_{i_{k}})^{\top},$$
and
$$g_{k}(A_{k}x)=\delta_{[\bm0_{k-1},\bm1_{i_{k}-1}]\times [-\infty,s] \times [\bm0_{n-i_{k}},\bm1_{n-i_{k}}]}(A_{k}x).$$
The conjugate function is
\begin{equation}
	g_{k}^*(y)=\sigma_{[\bm0_{i_{k}-1},\bm1_{i_{k}-1}]\times [-\infty,s] \times [\bm0_{n-i_{k}},\bm1_{n-i_{k}}]}(y).
\end{equation}
Again, since $A_k$ is a rank-one update of the identity matrix, the Sherman--Morrison formula gives
\[
A_{k}^{-1}
= \left(\bm I_n + e_{i_{k}}(\bm1_{n}-e_{i_{k}})^{\top}\right)^{-1}
= \bm I_n - \frac{e_{i_{k}}(\bm1_{n}-e_{i_{k}})^{\top}}{1+(\bm1_{n}-e_{i_{k}})^{\top}e_{i_{k}}}=\bm I_n - e_{i_{k}}(\bm1_{n}-e_{i_{k}})^{\top}.
\]
Substituting $[c_{l},c_{u}]=[\bm0_{i_{k}-1},\bm1_{i_{k}-1}]\times [-\infty,s] \times [\bm0_{n-i_{k}},\bm1_{n-i_{k}}]$ and $c_{e}=0$ into (\ref{conju}) and
using the preconditioner $P_{k}=A_{k}^{\top}D_{k}A_{k}$, the optimality condition gives
\begin{align*}
	y^k
	=
	\Gamma_k^{-1}
	\Big(
	&
	\left\{a^k_{[1:i_k-1]}
	-
	\mathtt{clip}\left(a^k_{[1:i_k-1]},\bm 0_{i_k-1},\bm 1_{i_k-1}\right)\right\} \\
	&\times
	\left[a^k_{i_k}-s\right]_+\times
		\left\{a^k_{[i_k+1:n]}
	-
	\mathtt{clip}\left(a^k_{[i_k+1:n]},\bm 0_{n-i_k},\bm 1_{n-i_k}\right)\right\}
	\Big).
\end{align*}
\section{Subspace acceleration and conjugate momentum}\label{sec5}
Although the diagonal preconditioning strategy improves the scaling of the proximal gradient step and preserves the closed-form structure of the dual subproblem, its capability remains limited. In particular, diagonal preconditioning only performs coordinate-wise scaling and therefore cannot fully capture the coupling between variables or the richer curvature information of the objective function. As a consequence, the practical acceleration obtained from diagonal scaling alone may still be insufficient, especially for ill-conditioned problems.

To further enhance the performance of the algorithm, we incorporate a subspace acceleration mechanism. Within this framework, two key questions naturally arise. The first concerns how to construct a subspace that effectively captures useful curvature and descent information from past iterates. The second concerns how to design an efficient subspace model so that the resulting subproblem can be solved rapidly while maintaining good approximation quality. These issues will be addressed in the following subsections.
\subsection{Selection of subspace and approximate model}
To exploit historical information while keeping the computational cost low, we construct a low-dimensional subspace that captures useful search directions from recent iterations. In particular, provided that \(v_k\) and \(s_k\) are linearly independent, for $k\geq1$ we define the two-dimensional subspace

 $$\mathcal{L}_{k}=\mathtt{span}\{v_{k},s_k\},$$
where $v_k$ represents the current preconditioned proximal gradient direction and $s_k$ incorporates information from the previous step. Specifically, the direction $s_k$ is defined as

\begin{equation}\label{s_k}
	s_{k}:=\left\{
	\begin{aligned}
		&d_{k-1}, & x^{k}+d_{k-1}\in\mathrm{dom}(g \circ A), \\
		&\mathrm{\Pi}_{\mathrm{dom}(g \circ A)}^{P_{k}}(x^{k}+d_{k-1})-x^{k}, &  \textrm{otherwise}.
	\end{aligned}
	\right.
\end{equation}
\begin{remark}
	If $\mathrm{dom}(g \circ A)=\mathbb{R}^{n}$, then $s_{k}=d_{k-1}$. When $\mathrm{dom}(g \circ A)\neq\mathbb{R}^{n}$, as in constrained optimization problems, the direction $d_{k-1}$ may become infeasible at the point $x^{k}$. In this case, if $P_{k}= \bm I_{n}$, then $s_{k}$ reduces to $\hat{s}_{k}$ in \cite[Eq.~(17)]{LLL2026}. However, computing the Euclidean projection $\mathrm{\Pi}_{\mathrm{dom}(g \circ A)}(x^{k}+d_{k-1})$ can be expensive. Instead, when $d_{k-1}$ is infeasible we set
	 $$s_{k}=\mathrm{\Pi}_{\mathrm{dom}(g \circ A)}^{P_{k}}(x^{k}+d_{k-1})-x^{k},$$ under our preconditioning framework. 
\end{remark}
\par The choice of the direction $v_k$ is also crucial. Here we define $v_k$ as the preconditioned proximal gradient direction at $x^{k}$, which is obtained by solving the following subproblem:
\begin{equation}\label{vk}
	\min\limits_{v\in\mathbb{R}^{n}}\nabla f(x^{k})^{\top}v+g(Ax^{k}+Av)+\frac{1}{2}\|v\|^{2}_{P_{k}}.
\end{equation}

Having constructed the subspace $\mathcal{L}_k$, we next define the corresponding subspace model used to refine the search direction. 
Restricting the step to $\mathcal{L}_k$, we consider the following subspace proximal Newton subproblem:
\begin{equation}\label{sn}
	\begin{aligned}
		\min_{d\in \mathcal{L}_{k}}\nabla f(x^{k})^{\top}d+g(Ax^{k}+Ad)+\frac{1}{2}\|d\|^{2}_{H_{k}},     
	\end{aligned}
\end{equation}
where $H_{k}:=\nabla^{2}f(x^{k})$.
Since $\mathcal{L}_k=\mathrm{span}\{v_k,s_k\}$ is two-dimensional, any $d\in\mathcal{L}_k$ can be written as $d=G_k\alpha$, where $G_k=[v_k,s_k]$ and $\alpha\in\mathbb{R}^2$. Substituting this representation into \eqref{sn} yields the equivalent two-dimensional optimization  problem
\begin{equation}\label{subsp}
	\begin{aligned}
		\min_{\alpha\in\mathbb{R}^{2}}\nabla f(x^{k})^{\top}G_{k}\alpha+g(Ax^{k}+AG_{k}\alpha)+\frac{1}{2}\|\alpha\|^{2}_{Q_{k}}   ,
	\end{aligned}
\end{equation}
where $Q_{k}=\begin{bmatrix}
	v_{k}^{\top}H_{k}v_{k}   & &~ & v_{k}^{\top}H_{k}s_{k} \\
	
	v_{k}^{\top}H_{k}s_{k}& &~  &s_{k}^{\top}H_{k}s_{k}
\end{bmatrix}$.
\subsection{Conjugate basis of subspace}
Although problem \eqref{subsp} is only two-dimensional, obtaining its exact solution may still be nontrivial due to the presence of the nonsmooth term $g(Ax^k+AG_k\alpha)$. In many practical situations, computing the exact minimizer is unnecessary and may introduce additional computational overhead. Therefore, instead of solving \eqref{subsp} exactly, we aim to construct an efficient approximation of the minimizer.

To this end, we exploit the structure of the subspace model and perform optimization along carefully chosen directions. In particular, by transforming the basis of the subspace into a conjugate basis with respect to the $H_k$-inner product, the quadratic term becomes diagonal.

Specifically, we orthogonalize $s_k$ with respect to $v_k$ under the $H_k$-inner product and define
\begin{equation}
	\tilde{s}_k = s_{k} - \frac{s_{k}^{\top}H_kv_k}{v_{k}^{\top}H_kv_k}  v_k.
\end{equation}

With this construction we have $v_k^{\top}H_k\tilde{s}_k=0$. Consequently, problem \eqref{subsp} can be rewritten in the equivalent form
\begin{equation}\label{subss}
	\begin{aligned}
		\min_{\alpha\in\mathbb{R}^{2}}\nabla f(x^{k})^{\top}\tilde{G}_{k}\alpha+g(Ax^{k}+A\tilde{G}_{k}\alpha)+\frac{1}{2}\|\alpha\|^{2}_{\tilde{Q}_{k}}   ,
	\end{aligned}
\end{equation}
where $\tilde{G}_{k}=[v_{k},\tilde{s}_{k}]$, and $\tilde{Q}_{k}=\begin{bmatrix}
	v_{k}^{\top}H_{k}v_{k}   & &~ & 0 \\
	
0& &~  &\tilde{s}_{k}^{\top}H_{k}\tilde{s}_{k}
\end{bmatrix}$.
Instead of solving \eqref{subss} directly, we adopt an inexact strategy by solving two one-dimensional subproblems. The first subproblem optimizes along the direction $v_k$:
\begin{equation}\label{subss1}
	\begin{aligned}
		\min_{\alpha_{1}\in\mathbb{R}}\nabla f(x^{k})^{\top}v_{k}\alpha_{1} +g(Ax^{k}+\alpha_{1}Av_{k})+\frac{1}{2}v_{k}^{\top}H_{k}v_{k}(\alpha_{1})^{2}.
	\end{aligned}
\end{equation}
The second subproblem optimizes along the orthogonalized direction $\tilde{s}_k$:
\begin{equation}\label{subss2}
	\begin{aligned}
		\min_{\alpha_{2}\in\mathbb{R}}\nabla f(x^{k})^{\top}\tilde{s}_{k}\alpha_{2} +g(Ax^{k}+\alpha_{2}A\tilde{s}_{k})+\frac{1}{2}\tilde{s}_{k}^{\top}H_{k}\tilde{s}_{k}(\alpha_{2})^{2}.
	\end{aligned}
\end{equation}
Let $\alpha_1^k$ and $\alpha_2^k$ denote the minimizers of \eqref{subss1} and \eqref{subss2}, respectively. We then define the intermediate direction
$$\tilde{d}_{k}=\alpha_{1}^{k}v_{k}+\alpha_{2}^{k}\tilde{s}_{k}.$$

To further refine the search direction, we perform an additional one-dimensional refinement along $\tilde{d}_k$ by solving the one-dimensional subproblem:
\begin{equation}\label{subss3}
	\begin{aligned}
		\min_{\alpha_{3}\in\mathbb{R}}\nabla f(x^{k})^{\top}\tilde{d}_{k}\alpha_{3} +g(Ax^{k}+\alpha_{3}A\tilde{d}_{k})+\frac{1}{2}\tilde{d}_{k}^{\top}H_{k}\tilde{d}_{k}(\alpha_{3})^{2}.
	\end{aligned}
\end{equation}
Let $\alpha_3^k$ be the minimizer of \eqref{subss3}. The final search direction is then given by
$$d_{k}=\alpha_{3}^{k}\tilde{d}_{k}.$$
 Note that $v_{k}^{\top}H_{k}\tilde{s}_{k}=0$ and $\tilde{d}_{k}=\alpha_{1}^{k}v_{k}+\alpha_{2}^{k}\tilde{s}_{k}$, then $$\tilde{d}_{k}^{\top}H_{k}\tilde{d}_{k}=\nm{\alpha_{1}^{k}v_{k}}_{H_{k}}^{2}+\nm{\alpha_{2}^{k}\tilde{s}_{k}}_{H_{k}}^{2}.$$
 \begin{remark}
 	We adopt $d_k$ as the final search direction rather than $\tilde{d}_k$ for two reasons. First, in constrained problems the additional scaling step ensures that the resulting direction remains feasible. Second, solving \eqref{subss3} provides an adaptive stepsize along $\tilde{d}_k$, which improves the stability and effectiveness of the search direction.
 \end{remark}
 \par It remains to compute the Hessian--vector products $H_{k}v_{k}$ and $H_{k}\tilde{s}_{k}$. 
 To avoid explicitly forming the Hessian matrix, we approximate the Hessian--vector products using finite differences of gradients. Specifically, we use
\begin{equation}\label{fd}
	H_{k}v\approx H_k(v):=\frac{1}{\epsilon}(\nabla f(x^{k}+\epsilon v)-\nabla f(x^{k})),~v\in\{v_{k},\tilde{s}_{k}\}.
\end{equation}
Based on this approximation, the one-dimensional subproblems \eqref{subss1} and \eqref{subss2} can be reformulated as
\begin{equation}\label{subs1}
	\begin{aligned}
		\min_{\alpha_{1}\in\mathbb{R}}\nabla f(x^{k})^{\top}v_{k}\alpha_{1} +g(Ax^{k}+\alpha_{1}Av_{k})+\frac{q_k(v_{k})}{2}\nm{v_{k}}^{2}(\alpha_{1})^{2}   ,
	\end{aligned}
\end{equation}
and 
\begin{equation}\label{subs2}
	\begin{aligned}
		\min_{\alpha_{2}\in\mathbb{R}}\nabla f(x^{k})^{\top}\tilde{s}_{k}\alpha_{2} +g(Ax^{k}+\alpha_{2}A\tilde{s}_{k})+\frac{q_k(\tilde{s}_{k})}{2}\nm{\tilde{s}_{k}}^{2}(\alpha_{2})^{2},
	\end{aligned}
\end{equation}
where $q_{k}(v)\approx v^{\top}H_k(v)/\nm{v}^{2},~v\in\{v_{k},\tilde{s}_{k}\}$.
Denote $$\tilde{d}_{k}=\alpha_{1}^{k}v_{k}+\alpha_{2}^{k}\tilde{s}_{k},$$
where $\alpha_{1}^{k}$ and $\alpha_{2}^{k}$ are the minimizers of (\ref{subs1}) and (\ref{subs2}), respectively.
To further refine the search direction, we solve an additional one-dimensional subproblem along $\tilde d_k$:
\begin{equation}\label{subs3}
	\begin{aligned}
		\min_{\alpha_{3}\in\mathbb{R}}\nabla f(x^{k})^{\top}\tilde{d}_{k}\alpha_{3} +g(Ax^{k}+\alpha_{3}A\tilde{d}_{k})+\frac{q_{k,d}}{2}(\alpha_{3})^{2},
	\end{aligned}
\end{equation}
where
\begin{equation}\label{qkd}
	q_{k,d}:=q_k(v_{k})\nm{\alpha_{1}^{k}v_{k}}^{2}+q_k(\tilde{s}_{k})\nm{\alpha_{2}^{k}\tilde{s}_{k}}^{2}
\end{equation} 
The final search direction is then defined as
\begin{equation}\label{dk}
	d_{k}:=\alpha_{3}^{k}\tilde{d}_{k},
\end{equation} 
where $\alpha_{3}^{k}$ is the minimizer of \eqref{subs3}.
\vspace{0.2cm}
\par The remaining question is whether the obtained search direction $d_k$ satisfies the sufficient descent condition. Before analyzing this property, we first establish the following auxiliary lemma.
\vspace{-0.2cm}
\begin{lemma}\label{ll2}
Let $h:\mathbb{R}^{n}\rightarrow\mathbb{R}\cup\{+\infty\}$ be a proper convex and lower semicontinuous function, which is not necessarily differentiable. Assume that $x^{*}$ is the minimizer of
	\begin{equation}\label{op}
		\min\limits_{x\in\mathbb{R}^{n}}h(x)+\frac{1}{2}\nm{x}^{2}_{P},
	\end{equation}
where $P\succ0$. Then 
\begin{equation}
	h(x^*)-h(0)\leq-\nm{x^{*}}^{2}_{P}
\end{equation}
\end{lemma}
\begin{proof}
	Since $P\succ0$ and $h$ is convex, the optimality condition of (\ref{op}) gives
	$$0\in Px^*+\partial h(x^*).$$
	Combining this with the convexity of $h$ yields
	$$h(0)-h(x^*)\geq (-Px^*)^{\top}(0-x^*).$$
	Rearranging the terms gives the desired result.
 \end{proof}

We are now ready to provide a sufficient condition under which the direction $d_k$ is a descent direction.

\begin{proposition}\label{suco}
	Assume that there exist constants $0<c_{1}\leq c_{2}$ and $c_{3}>0$ such that
	 $c_{1}\leq q_{k}(v_{k})\leq c_{2}$, $q_{k}(\tilde{s}_{k})\geq c_{1}$ and $P_{k}\succeq c_{3}\bm{I}_{n}$ in (\ref{subs1}), (\ref{subs2}) and (\ref{vk}) for all $k$. Then, the search direction $d_{k}$ defined in (\ref{dk}) satisfies the following conditions:  
	\begin{equation}\label{con1}
		\nabla f(x^{k})^{\top}d_{k}+g(Ax^{k}+Ad_{k})-g(Ax^{k})\leq-\frac{c_1}{2}\nm{d_{k}}^{2}.
		\end{equation}
	\begin{equation}\label{con2}
		\nabla f(x^{k})^{\top}d_{k}+g(Ax^{k}+Ad_{k})-g(Ax^{k})\leq\frac{\min\{1,c_{3}/c_{2}\}}{4}(\nabla f(x^{k})^{\top}v_{k}+g(Ax^{k}+Av_{k})-g(Ax^{k})).
	\end{equation}
\end{proposition}
\begin{proof}
	By Lemma \ref{ll2} and the definition of $d_k$, we obtain $$\nabla f(x^{k})^{\top}d_{k}+g(Ax^{k}+Ad_{k})-g(Ax^{k})\leq-q_{k,d}\cdot(\alpha^{k}_{3})^{2}.$$ On the other hand, by the fact $\nm{a}^{2}+\nm{b}^{2}\geq1/2\nm{a+b}^{2}$ and the definition of $q_{k,d}$, we have $$q_{k,d}\cdot(\alpha_{3})^{2}=q_k(v_{k})\nm{\alpha^{k}_{3}\alpha_{1}^{k}v_{k}}^{2}+q_k(\tilde{s}_{k})\nm{\alpha^{k}_{3}\alpha_{2}^{k}\tilde{s}_{k}}^{2}\geq \frac{c_1}{2}\nm{\alpha^{k}_{3}(\alpha_{1}^{k}v_{k}+\alpha_{2}^{k}\tilde{s}_{k})}^{2}=\frac{c_1}{2}\nm{{d}_{k}}^{2}.$$
	Therefore, 
	$$\nabla f(x^{k})^{\top}d_{k}+g(Ax^{k}+Ad_{k})-g(Ax^{k})\leq-\frac{c_1}{2}\nm{d_{k}}^{2},$$
	which proves \eqref{con1}.
	\par Next we prove \eqref{con2}. Since $\alpha_{3}^{k}$ and $\alpha_{1}^{k}$ are the minimizers of \eqref{subs3} and \eqref{subs1}, respectively, we have
	\begin{alignat*}{2}
		&~\quad &&\nabla f(x^{k})^{\top}d_{k}+g(Ax^{k}+Ad_{k})-g(Ax^{k})\\
		&\overset{\mathclap{(\mathrm{set} ~\alpha_{3}=\frac{1}{2}~\mathrm{in}~ (\ref{subs3}))}}{\leq}\quad &&-\frac{q_{k,d}}{2}(\alpha^{k}_{3})^{2}+\frac{1}{2}\nabla f(x^{k})^{\top}\tilde{d}_{k}+g(Ax^{k}+\frac{1}{2}A\tilde{d}_{k})-g(Ax^{k})+\frac{q_{k,d}}{8}\\
		&\overset{\mathclap{(\mathrm{convexity~of~}g)}}{\leq} \quad &&\frac{1}{2}\nabla f(x^{k})^{\top}(\alpha_{1}^{k}v_{k}+\alpha_{2}^{k}\tilde{s}_{k})+ \frac{1}{2}g(Ax^{k}+\alpha_{1}^{k}Av_{k}) +  \frac{1}{2}g(Ax^{k}+\alpha_{2}^{k}A\tilde{s}_{k})\\
		&~~~\quad &&-g(Ax^{k})+\frac{1}{8}(q_k(v_{k})\nm{\alpha_{1}^{k}v_{k}}^{2}+q_k(\tilde{s}_{k})\nm{\alpha_{2}^{k}\tilde{s}_{k}}^{2})\\
		&=\quad \quad \quad\quad\quad&&\frac{1}{2}(\nabla f(x^{k})^{\top}v_{k}\alpha_{1}^{k}+g(Ax^{k}+\alpha_{1}^{k}Av_{k})-g(Ax^{k})+\frac{1}{4}q_k(v_{k})\nm{\alpha_{1}^{k}v_{k}}^{2})\\
		&~~~\quad &&+\underbrace{\frac{1}{2}(\nabla f(x^{k})^{\top}\tilde{s}_{k}\alpha_{2}^{k}+g(Ax^{k}+\alpha_{2}^{k}A\tilde{s}_{k})-g(Ax^{k})+\frac{1}{4}q_k(\tilde{s}_{k})\nm{\alpha_{2}^{k}\tilde{s}_{k}}^{2})}_{\leq0~(\mathrm{by~Lemma}~\ref{ll2})}\\
		&\leq\quad &&\frac{1}{2}(\nabla f(x^{k})^{\top}v_{k}\alpha_{1}^{k}+g(Ax^{k}+\alpha_{1}^{k}Av_{k})-g(Ax^{k})+\frac{1}{2}q_k(v_{k})\nm{\alpha_{1}^{k}v_{k}}^{2})\\
		&\overset{\mathclap{(\mathrm{any} ~0\leq\alpha_{1}\leq1)}}{\leq}\quad &&\frac{1}{2}(\nabla f(x^{k})^{\top}v_{k}\alpha_1+g(Ax^{k}+\alpha_{1}Av_{k})-g(Ax^{k})+\frac{1}{2}q_k(v_{k})\nm{v_{k}}^{2}(\alpha_{1})^{2} )\\
		&\overset{\mathclap{(\mathrm{convexity~of~}g)}}{\leq}\quad &&\frac{1}{2}(\nabla f(x^{k})^{\top}v_{k}\alpha_1+\alpha_{1}(g(Ax^{k}+Av_{k})-g(Ax^{k}))+\frac{c_{2}}{2c_{3}}\nm{v_{k}}_{P_{k}}^{2}(\alpha_{1})^{2} )\\
		&=\quad &&\frac{\alpha_{1}}{2}(\nabla f(x^{k})^{\top}v_{k}+g(Ax^{k}+Av_{k})-g(Ax^{k})+\frac{c_{2}\alpha_{1}}{2c_{3}}\nm{v_{k}}_{P_{k}}^{2})\\
		&\overset{\mathclap{(\mathrm{any} ~0\leq\alpha_{1}\leq\min\{1,c_3/c_{2}\})}}{\leq}\quad &&\frac{\alpha_{1}}{2}(\nabla f(x^{k})^{\top}v_{k}+g(Ax^{k}+Av_{k})-g(Ax^{k})+\frac{1}{2}\nm{v_{k}}_{P_{k}}^{2})\\
		&\overset{\mathclap{(\mathrm{by~Lemma}~\ref{ll2})}}{\leq}\quad &&\frac{\alpha_{1}}{4}(\nabla f(x^{k})^{\top}v_{k}+g(Ax^{k}+Av_{k})-g(Ax^{k}))\\
		&\leq\quad &&\frac{\min\{1,c_{3}/c_{2}\}}{4}(\nabla f(x^{k})^{\top}v_{k}+g(Ax^{k}+Av_{k})-g(Ax^{k})),
	\end{alignat*}
\end{proof}
where the last inequality follows by $P_k\succeq c_3I$ and choosing $\alpha_{1} =\min\{1,c_3/c_{2}\}$. The proof is completed.
\section{Preconditioned proximal gradient method with conjugate momentum}\label{sec6}

To guarantee the sufficient condition stated in Proposition \ref{suco}, 
for $v\in\{v_{k},\tilde{s}_{k}\}$ we define
\begin{equation}\label{bb}
	q_k(v):=\left\{
	\begin{aligned}
		&\max\left\{c_{1},\min\left\{\frac{v^{\top}H_k(v)}{\nm{v}^{2}}, c_{2}\right\}\right\}, & v^{\top}H_k(v)&>0, \\
		&\max\left\{c_{1},\min\left\{\frac{\nm{H_{k}(v)}}{\nm{v}}, c_{2}\right\}\right\}, &v^{\top}H_k(v)&<0, \\
		& c_{1}, &  v^{\top}H_k(v)&=0,
	\end{aligned}
	\right.
\end{equation}
where $H_{k}(v)$ is defined in (\ref{fd}), $c_{1}$ and $c_{2}$ are the positive constants introduced in Proposition \ref{suco}.

The complete preconditioned proximal gradient method with conjugate momentum is described as follows.
\vspace{0.2cm}

\begin{algorithm}[H]  
	\caption{\small Preconditioned proximal gradient method with conjugate momentum ($\mathrm{P^{2}GM}$\_CM)}\label{ppg} 
	\begin{algorithmic}[1]
		\REQUIRE{$x^{0}\in\mathrm{dom}(g\circ A)$, $0<c_{1}\leq c_{2}$, $0<c_{3}\leq c_{4}$, $\sigma,\gamma\in(0,1)$}
		\FOR{$k=0,\cdots$}
		\STATE{Update $c_{3}\bm{I}_{n}\preceq P_{k}\preceq c_{4}\bm{I}_{n}$} 
		\STATE{Compute $v_{k}$ as the solution of (\ref{vk})} 
		\IF{$v_{k}=0$}
		\RETURN{$x^{k}$  }
		\ELSE{
			\IF{$k=0$}
			\STATE{Set $d_{k}=v_{k}$}
			\ELSE{
			\STATE{Compute $s_{k}$ as in (\ref{s_k})}
			\STATE{Update $H_{k}(v_{k})$ and $q_{k}(v_{k})$ as in (\ref{fd}) and (\ref{bb}), respectively}
				\STATE{Update $$\tilde{s}_{k}:=s_{k}-\frac{s_{k}^{\top}H_{k}(v_{k})}{q_{k}(v_{k})\nm{v_{k}}^{2}}v_{k}$$}
				\STATE{Update $H_{k}(\tilde{s}_{k})$ and $q_{k}(\tilde{s}_{k})$ as in (\ref{fd}) and (\ref{bb}), respectively}
				\STATE{Compute $\alpha^{k}_{1}$ and $\alpha^{k}_{2}$ as the minimizers of (\ref{subs1}) and (\ref{subs2}), respectively}
				\STATE{Update $\tilde{d}_{k}:=\alpha^{k}_{1}v_{k}+\alpha^{k}_{2}\tilde{s}_{k}$}
				\STATE{Update $q_{k,d}$ as in (\ref{qkd})}
				\STATE{Compute $\alpha^{k}_{3}$ as the minimizer of (\ref{subs3})}
				\STATE{Update $d_{k}:=\alpha^{k}_{3}\tilde{d}_{k}$}
			}
		\ENDIF
		\STATE{Compute the stepsize $t_{k}\in(0,1]$ in the following way:
			\begin{align*}
			t_{k}:=\max\big\{\gamma^{j}:j\in\mathbb{N},~&F\left(x^{k}+\gamma^{j}d_{k}\right)-F(x^{k})\\
			&\leq \sigma\gamma^{j}(\nabla f(x^{k})^{\top}d_{k} + g(Ax^{k}+ Ad_{k})-g(Ax^{k})).\big\}	
			\end{align*}
		}
		\STATE{Update $x^{k+1}:=x^{k}+t_{k}d_{k}$}
		}
		\ENDIF
		\ENDFOR
	\end{algorithmic}
\end{algorithm}
\begin{remark}
	Lines 3, 10, 14, and 17 contribute to the main computational cost of Algorithm \ref{ppg}, since each of these steps requires solving a subproblem. 
	Fortunately, due to the specific choice of $P_{k}$, the subproblems appearing in Lines 3 and 10 admit closed-form solutions. 
	The remaining challenge lies in solving the three one-dimensional subproblems (\ref{subs1}), (\ref{subs2}), and (\ref{subs3}). 
	In the next subsection, we present efficient methods for solving these one-dimensional subproblems.
\end{remark}

\subsection{Methods for one-dimensional subproblems}
The subproblems arising in Algorithm \ref{ppg} reduce to several one-dimensional optimization problems. 
Depending on the structure of the objective function and the constraints, these problems can be categorized into two types: constrained quadratic optimization problems and $\ell_{1}$-regularized optimization problems. 
In the following, we present specialized solution strategies for each case.
\subsubsection{Constrained optimization problems}
The one-dimensional subproblem takes the form:
$$\min\limits_{x^{k}+td_{k}\in\Omega}a{t}^{2}+bt,$$
where $a>0$ and $b\in\mathbb{R}$ are constants.
Note that there exist $t_{l}$ and $t_{u}$ such that
$$\{t:x^{k}+td_{k}\in\Omega\}=[t_{l},t_{u}].$$
The optimal solution is therefore given by
 $$t_{k}=\min\{\max\{t_{l},-\frac{b}{2a}\},t_{u}\}.$$

We summarize the procedure in the following algorithm.

\begin{algorithm}[H]  
	\caption{Method for one-dimensional constrained optimization problem}
	\begin{algorithmic}[1]
		\REQUIRE{$x^{k},d_{k}\in\mathbb{R}^{n}$, $a,b$, $\Omega$ }
		\STATE{Compute $t_k=-\frac{b}{2a}$}
		\IF{$x^{k}+t_{k}d_{k}\in\Omega$}
		\RETURN{$t_{k}$}
		\ELSE{
			\RETURN{$t_{k}=-\frac{b}{2a}\max\{t: x^{k}+t(-\frac{b}{2a}d_{k})\in\Omega\}$}
		}
		\ENDIF
	\end{algorithmic}
\end{algorithm}
\subsubsection{$\ell_{1}$ regularized problems}
Another type of subproblem arising in Algorithm \ref{ppg} involves $\ell_{1}$ regularization and takes the form:
$$\min h(t):= a{t}^{2}+bt+\|v+td\|_{1},$$
where $a>0,b\in\mathbb{R}$ and $v,d\in\mathbb{R}^{m}$. The optimality condition for this problem is
$$0\in\partial{h}(t):= 2at+b+\sum\limits_{i\in[m]}\partial|v_{i}+td_{i}|,$$
where \begin{equation}\label{partial}
	\partial|v_{i}+td_{i}|:=\left\{
	\begin{aligned}
		&d_{i}, & v_{i}+td_{i}&>0, \\
		&[-|d_{i}|,|d_{i}|], &v_{i}+td_{i}&=0, \\
		& -d_{i}, &  v_{i}+td_{i}&<0.
	\end{aligned}
	\right.
\end{equation}
Since $h$ is convex, the subdifferential satisfies $$\partial^{+}h(t_{1})\leq \partial^{-}h(t_{2}),\quad t_{1}<t_{2}.$$ 
Moreover, we have $$\partial h(t)\subset[2at+b-D,2at+b+D],$$ where $D:=\sum_{i\in[m]}|d_{i}|$. Based on this property, we develop a partitioning method with the initial interval
$$
\left[\frac{-D-b}{2a},\frac{D-b}{2a}\right].
$$

\begin{algorithm}[H]  
	\caption{Partitioning method for one-dimensional $\ell_{1}$ regularized problem}\label{part} 
	\begin{algorithmic}[1]
		\REQUIRE{$L=\frac{-D-b}{2a},U=\frac{D-b}{2a},S=[L,U]\cap\{-\frac{v_{i}}{d_{i}}:d_{i}\neq0,i\in[m]\}\cup\{L,U\}$, $Bool=\mathtt{True}$}
		\WHILE{$Bool$}
		\IF{$|S|\leq2$}
		\STATE{$Bool=\mathtt{False}$}
		\STATE{Compute  $s = -(b+\sum_{i\in[m]}\mathtt{sign}(v_{i}+\frac{L+U}{2}d_{i})d_{i})/(2a)$}
		\RETURN{$\max\left\{L,\min\left\{s, U\right\}\right\}$}
		\ELSE{
			\STATE{Select $s_{median}$ as the median of $S$}
			\STATE{Compute the subdifferential $\partial h(s_{median})$} 
			\IF{$0\in\partial h(s_{median})$}
			\STATE{$Bool=\mathtt{False}$}
			\RETURN{$s_{median}$}
			\ELSE{
				\IF{$\partial^{+} h(s_{median})<0$}
				\STATE{Update $L=s_{median}$}
				\STATE{Update $S=S_{[s\geq s_{median}]}$}
				\ELSE{
					\STATE{Update $U=s_{median}$}
					\STATE{Update $S=S_{[s\leq s_{median}]}$}
				}
				\ENDIF
			}
			\ENDIF
		}
		\ENDIF
		\ENDWHILE
	\end{algorithmic}
\end{algorithm}

\subsection{Convergence analysis}
In this section we analyze the convergence properties of the proposed algorithm.
We first show that the stepsize produced by the line-search procedure admits a uniform lower bound.
\vspace{1mm}
\begin{lemma}
	Suppose that $f$ is $L$-smooth. The stepsize generated by Algorithm \ref{ppg} has a lower bound for $k\geq1$.
	\begin{equation}\label{lbt}
		t_{\min}:=\min\left\{{\gamma(1-\sigma)c_1}/{L},1\right\}.
	\end{equation} 
\end{lemma}
\vspace{1mm}
\begin{proof} It suffices to consider the case $t_k<1$, in which the backtracking procedure is activated. 
	In this situation the Armijo condition is violated for the trial stepsize $t_k/\gamma$, yielding
	\begin{equation}\label{E4.4}
		F\left(x^{k}+\frac{t_{k}}{\gamma}d_{k}\right)-F(x^{k})>\sigma\frac{t_{k}}{\gamma}(\nabla f(x^{k})^{\top}d_{k} + g(Ax^{k}+ Ad_{k})-g(Ax^{k})).
	\end{equation}
	Since $f$ is $L$-smooth, we have
	\begin{equation}\label{up}
		\begin{aligned}
			&~~~F\left(x^{k}+\frac{t_{k}}{\gamma}d_{k}\right)-F(x^{k})\\
			&\leq\frac{t_{k}}{\gamma}\nabla f(x^{k})^{\top}d_{k} + g(Ax^{k}+\frac{t_{k}}{\gamma}Ad_{k}) - g(Ax^{k})+ \frac{L}{2}\left\|\frac{t_{k}}{\gamma}d_{k}\right\|^{2}\\
			&\leq\frac{t_{k}}{\gamma}(\nabla f(x^{k})^{\top}d_{k} + g(Ax^{k}+ Ad_{k})-g(Ax^{k}))+\frac{L}{2}\left\|\frac{t_{k}}{\gamma}d_{k}\right\|^{2},
		\end{aligned}
	\end{equation}
	where the second inequality follows from the convexity of $g$ and the fact that $t_{k}/{\gamma}\in(0,1]$. Combining this inequality with \eqref{E4.4} gives
	$$(\sigma - 1)(\nabla f(x^{k})^{\top}d_{k} + g(Ax^{k}+ Ad_{k})-g(Ax^{k}))\leq\frac{Lt_{k}}{2\gamma}\left\|d_{k}\right\|^{2}.$$
	Using condition \eqref{con1}, we obtain
	\begin{equation}\label{et}
		t_{k}\geq\frac{\gamma(1-\sigma)c_{1}}{L}.
	\end{equation}
	Therefore $t_k \ge t_{\min}$, which completes the proof. \qed
\end{proof}

To establish global convergence, we impose the following standard assumption on the objective function.

\begin{assumption}\label{a2}
	For any $x^{0}\mathbb\in\mathrm{dom}F$, the level set $\mathcal{L}_{F}(x^0):=\{x:F(x)\leq F(x^0)\}$ is compact.
\end{assumption}

Under this assumption we can prove the global convergence of the proposed algorithm.

\begin{theorem}
	Suppose that Assumption \ref{a2} holds and $f$ is $L$-smooth. Let $\{x^{k}\}$ be the sequence generated by Algorithm \ref{ppg}. Then $\{x^k\}$ has at least one accumulation point, and any accumulation point $x^*$ is a stationary point.
\end{theorem}
\begin{proof}
	By the Armijo line search, we deduce that $\{F(x^{k})\}$ is monotone decreasing and 
	\begin{equation}\label{des}
		F(x^{k+1}) - F(x^{k})\leq \sigma t_{k} (\nabla f(x^{k})^{\top}d_{k}+g(Ax^{k}+ Ad_{k})-g(Ax^{k}))\leq -\frac{\min\{1,c_{3}/c_{2}\}}{4} \sigma t_{k}\nm{v_{k}}_{P_{k}}^{2},
	\end{equation}
where the last inequality follows by relation (\ref{con2}) and Lemma \ref{ll2}. Therefore $x^{k}\in\mathcal{L}_{F}(x^0)$ for all $k$, and hence $\{x^{k}\}$ has at least one accumulation point $x^*$ due to the compactness of $\mathcal{L}_{F}(x^0)$. 
In particular, there exists an infinite index set $\mathcal{K}$ such that
\[
\lim_{k\in\mathcal{K}}x^{k}=x^* .
\]
Moreover, since $F$ is lower semicontinuous and $\mathcal{L}_{F}(x^0)$ is compact, the sequence $\{F(x^{k})\}$ is bounded below. 
Together with the monotonicity of $\{F(x^{k})\}$, this implies that $\{F(x^{k})\}$ is a Cauchy sequence. Hence
$$\lim_{k\rightarrow\infty}F(x^{k+1})-F(x^{k})=0.$$
Combining this limit with \eqref{des} yields
\begin{equation}\label{lim}
	\lim_{k\rightarrow\infty}t_{k}\nm{v_{k}}_{P_k}^{2}=0.
\end{equation}
Together with \eqref{et} and the fact $P_{k}\succeq c_{3}\bm{I}_{n}$, we obtain
$$\mathop{\lim}\limits_{k\rightarrow\infty}v_{k}=0.$$
Since $P_{k}\succeq c_{3}\bm{I}_{n}$, we conclude that $x^{*}$ is a stationary point. \qed
\end{proof}

The above theorem guarantees that every accumulation point is stationary. 
Next, we further strengthen the result by establishing a linear convergence rate under the stronger assumption that the objective function is strongly convex.

\begin{theorem}
	Suppose that $f$ is $L$-smooth and strongly convex with modulus $\mu>0$. Let $\{x^{k}\}$ be the sequence generated by Algorithm \ref{ppg}. Then, for all $k\geq1$
	$$F(x^{k+1})-F(x^*)\leq \left(1-\frac{\sigma}{4}\min\left\{\frac{\mu}{c_{4}},1\right\}\min\left\{\frac{c_{3}}{c_{2}},1\right\}\min\left\{\frac{\gamma(1-\sigma)c_1}{L},1\right\}\right)(F(x^{k})-F(x^*)).$$
	
\end{theorem}
\begin{proof}
	By direct calculation, we have
	\begin{align*}
		F(x^{k}) - F(x^{*}) &\leq \max\limits_{x\in\mathbb{R}^{n}}\big\{\nabla f(x^{k})^{\top}(x^{k}-x)+g(Ax^{k})-g(Ax) - \frac{\mu}{2}\nm{x-x^k}^{2}\big\}\\
		&\leq \max\limits_{x\in\mathbb{R}^{n}}\big\{\nabla f(x^{k})^{\top}(x^{k}-x)+g(Ax^{k})-g(Ax) - \frac{\mu}{2c_{4}}\nm{x-x^k}_{P_k}^{2}\big\}\\
		&\leq\max\left\{\frac{c_{4}}{\mu},1\right\}\max\limits_{x\in\mathbb{R}^{n}}\big\{\nabla f(x^{k})^{\top}(x^{k}-x)+g(Ax^{k})-g(Ax) - \frac{1}{2}\nm{x-x^k}_{P_k}^{2}\big\}\\
		&=-\max\left\{\frac{c_{4}}{\mu},1\right\}\left(\nabla f(x^{k})^{\top}v_{k}+g(Ax^{k}+Av_{k})-g(Ax^{k})+\frac{1}{2}\nm{v_k}_{P_k}^{2}\right)\\
		&\leq -\max\left\{\frac{c_{4}}{\mu},1\right\}\left(\nabla f(x^{k})^{\top}v_{k}+g(Ax^{k}+Av_{k})-g(Ax^{k})\right)\\
		&\leq - 4\max\left\{\frac{c_{4}}{\mu},1\right\}\max\left\{\frac{c_{2}}{c_{3}},1\right\}(\nabla f(x^{k})^{\top}d_{k} + g(Ax^{k}+ Ad_{k})-g(Ax^{k})),
	\end{align*}
where the third inequality is due to the convexity of $g$ and the last inequality follows by (\ref{con2}).
Substituting this bound into line search condition gives
$$F(x^{k+1}) - F(x^{k})\leq -\frac{1}{4}\min\left\{\frac{\mu}{c_{4}},1\right\}\min\left\{\frac{c_{3}}{c_{2}},1\right\}\sigma t_{k}(F(x^{k}) - F(x^{*})).$$
The desired result follows by rearranging the inequality and adding $-F(x^*)$ to both sides.

\end{proof}

\section{Numerical experiments}\label{sec7}
In this section, we evaluate the empirical performance of the proposed method on several composite optimization problems. The goal of these experiments is to assess the efficiency and robustness of the proposed algorithm in comparison with several widely used first-order methods.

The tested algorithms are summarized as follows:

\begin{itemize}
	\item {FISTA\_bt}: FISTA with backtracking \cite{SGB2014};
	
	\item {FISTA\_bt\_rs}: FISTA with backtracking and gradient restart \cite{OC2015};
	
	\item {PDHG}: primal-dual hybrid gradient method \cite{HY2012}:
	\begin{align*}
		x^{k+1}&=x^{k} - \tau (\nabla f(x^{k})+A^{\top}y^{k}),\\
		\bar{x}^{k}&= x^{k+1}+\theta(x^{k+1}-x^{k}),\\
		y^{k+1}&=\mathrm{prox}_{\sigma g^{*}}(y^{k}+\sigma A\bar{x}^{k}),
	\end{align*}
where $\tau = 1 / L$, $\theta = 0.9$ and $\sigma = 4 / (\tau (1+\theta)^{2}\|A^{\top}A\|)$
	;
	\item $\mathrm{P^{2}GM}$\_M: the preconditioned proximal gradient method with momentum, obtained from Algorithm \ref{ppg} by removing Line 12, i.e., $\tilde{s}_{k}=s_{k}$;
	\item $\mathrm{P^{2}GM}$\_CM: preconditioned proximal gradient method with conjugate momentum, as described in Algorithm \ref{ppg}.
\end{itemize}
In $\mathrm{P^{2}GM}$\_M and $\mathrm{P^{2}GM}$\_CM, the preconditioning matrix is chosen as 
$P_k = \alpha_k P$, where $P$ is defined as in \eqref{p}. 
The scalar $\alpha_k$ is selected according to a Barzilai--Borwein
stepsize computed in the $P$-norm, which allows the scaling of the 
preconditioner to adapt to the local curvature of the objective function.

All numerical experiments were implemented in Python 3.7 and conducted on a personal computer equipped with an Intel Core i7-11390H processor (3.40 GHz) and 16 GB of RAM.
To evaluate the convergence behavior of the tested algorithms, we compute an approximation of the optimal objective value by running $\mathrm{P^{2}GM}$\_CM for a sufficiently large number of iterations. Denote this value by $\tilde F \approx F^*$. For each algorithm we then report the objective gap $F(x^k)-\tilde F$.

In the following subsections, we present numerical results on three representative applications: LASSO problems, simplex-constrained quadratic problems, and structured quadratic composite problems.

\subsection{LASSO problems}
We first consider the LASSO problem, which is a widely used benchmark in sparse signal recovery and machine learning. The problem is given by

\begin{equation}\label{lasso}
\min_{x \in \mathbb{R}^n}
\frac{1}{2}\|Ax-b\|^2 + \lambda \|x\|_1.	
\end{equation}
Here $A\in\mathbb{R}^{m\times n}$ is the data matrix, $b\in\mathbb{R}^m$ is the observation vector, and $\lambda>0$ is a regularization parameter controlling the sparsity of the solution.

The problem instance is synthetically generated to simulate a high-dimensional, noisy, and numerically ill-conditioned sparse regression environment.

\begin{itemize}
	
	\item \textbf{Dimensions:}
	We set the number of samples $m = 5000$ and the feature dimension $n = 500$.
	
	\item \textbf{Ill-Conditioned Design Matrix:}
The matrix $A$ is constructed using the singular value decomposition
	\[
	A = U \Sigma V^{\top}.
	\]
	To induce numerical stiffness, the singular values are logarithmically spaced in the interval $[10^{-3},1]$ and scaled by $\sqrt{m}$. This scaling ensures that the data term
	\[
	\frac{1}{m}A^{\top}A
	\]
	has a bounded spectral norm while maintaining a large condition number $\kappa \approx 10^{6}$. The orthogonal matrices $U$ and $V$ are generated from random Gaussian matrices via QR factorization.
	
	\item \textbf{Sparse Ground Truth:}
	The true solution $x_{\mathrm{true}}$ is generated with extreme sparsity ($p = 0.5\%$), resulting in approximately $2$--$3$ active nonzero elements drawn uniformly from $[1,2]$. This sparsity forces the optimization trajectory to frequently interact with the non-smooth boundaries induced by the $\ell_1$ regularization.
	
	\item \textbf{Noisy Observations:}
	The observation vector is generated as
	\[
	b = A x_{\mathrm{true}} + \varepsilon,
	\]
	where $\varepsilon$ is Gaussian noise with noise level $10^{-3}$. This introduces measurement noise and simulates realistic perturbations in the data.
	
\end{itemize}
\begin{figure}[H]
	\centering
	\subfigure[Iterations]
	{
		\begin{minipage}[H]{.45\linewidth}
			\centering
			\includegraphics[scale=0.4]{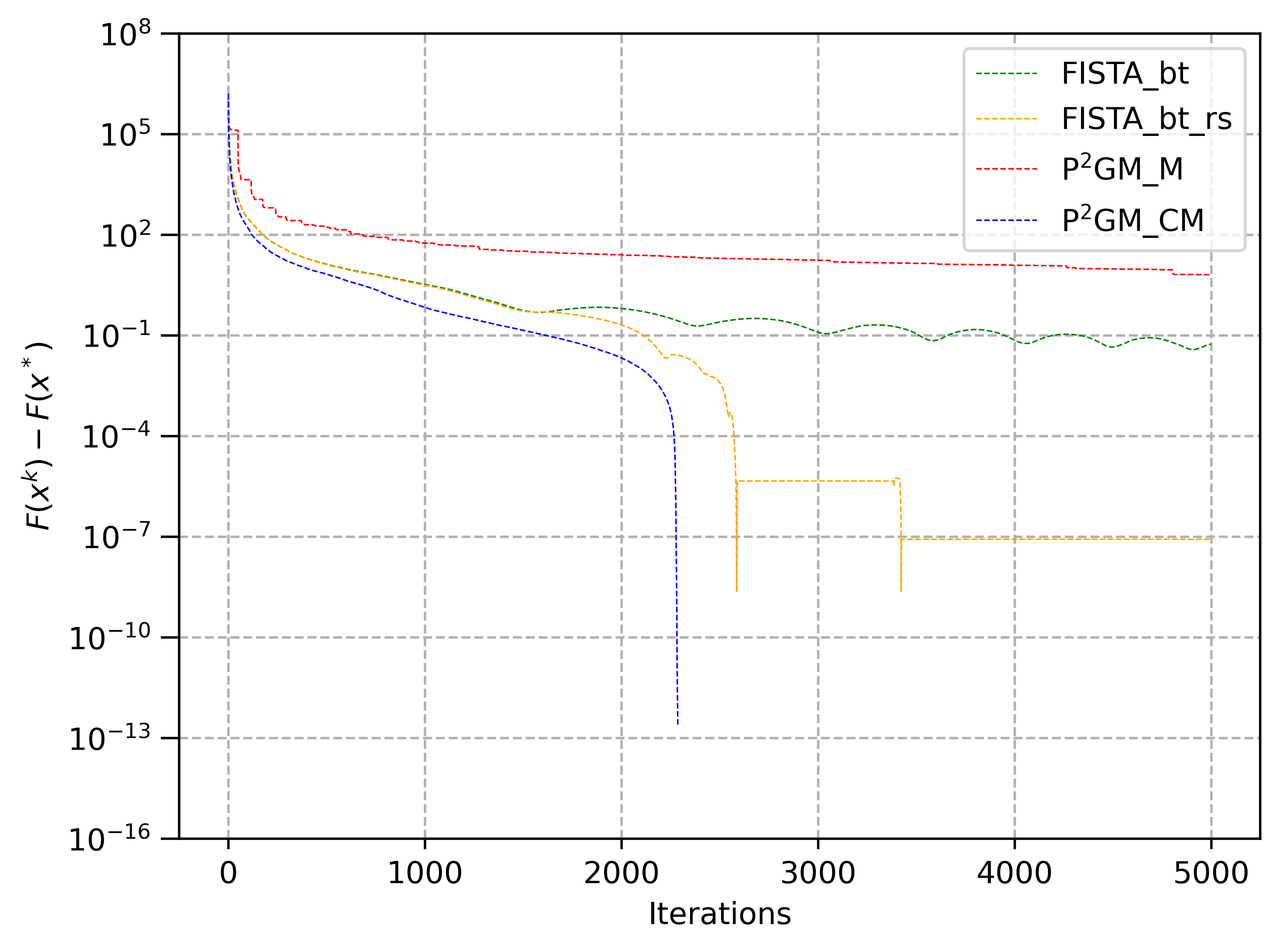} 
		\end{minipage}
	}
	\subfigure[CPU time]
	{
		\begin{minipage}[H]{.45\linewidth}
			\centering
			\includegraphics[scale=0.4]{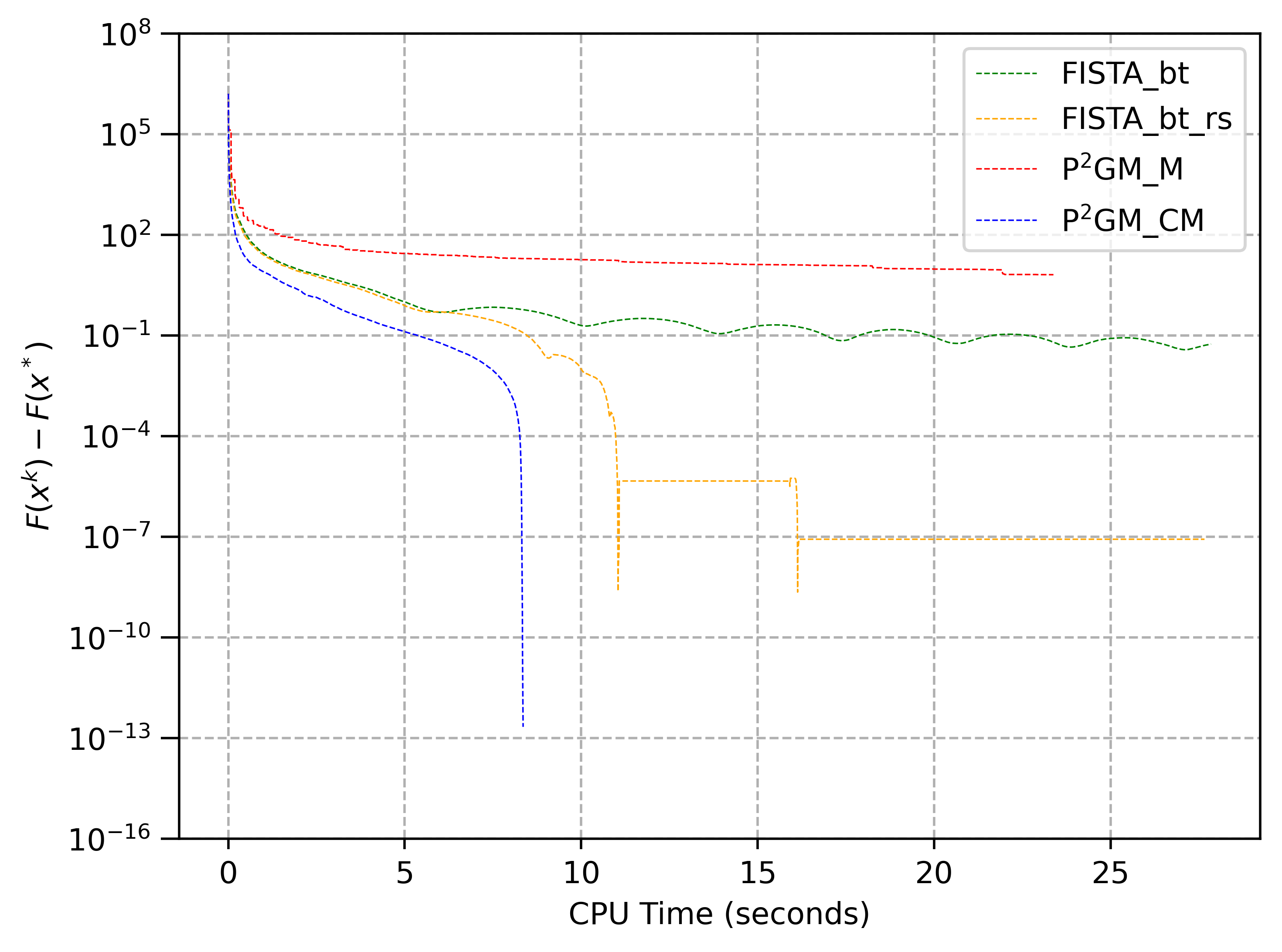} 
		\end{minipage}
	}
	\caption{Objective gaps w.r.t. iterations and CPU time for problem (\ref{lasso}) with $\lambda=10^{-4}$.}
	\label{f1}
\end{figure}

\subsection{Ill-conditioned quadratic problems with simplex constraint}

Next, we consider the minimization of a quadratic function over the unit simplex:
\begin{equation}\label{qp}
	\min_{x \in \Delta_n} f(x)=\frac12 x^\top Q x + c^\top x.
\end{equation}
This problem serves as a representative example of constrained optimization with a highly ill-conditioned objective function.

The problem instance is generated according to the following specifications:
\begin{itemize}
	
	\item \textbf{Dimension:}
	$n = 100$.
	
	\item \textbf{Condition Number:}
	\[
	\kappa =  5\times10^{5}.
	\]
	
	\item \textbf{Spectrum:}
	The eigenvalues of $Q$ are logarithmically spaced in the interval $[1,\kappa]$.
	
	\item \textbf{Matrix Construction:}
	The positive definite matrix $Q$ is constructed via spectral decomposition
	\[
	Q = U \Lambda U^\top ,
	\]
	where $U$ is a random orthogonal matrix generated from a Gaussian matrix via QR factorization and $\Lambda=\mathrm{diag}(\lambda_1,\ldots,\lambda_n)$ contains the prescribed eigenvalues.
	
	\item \textbf{Linear Term:}
	The vector $c$ is sampled from a standard Gaussian distribution.	
\end{itemize}
\begin{figure}[H]
	\centering
	\subfigure[Iterations]
	{
		\begin{minipage}[H]{.45\linewidth}
			\centering
			\includegraphics[scale=0.4]{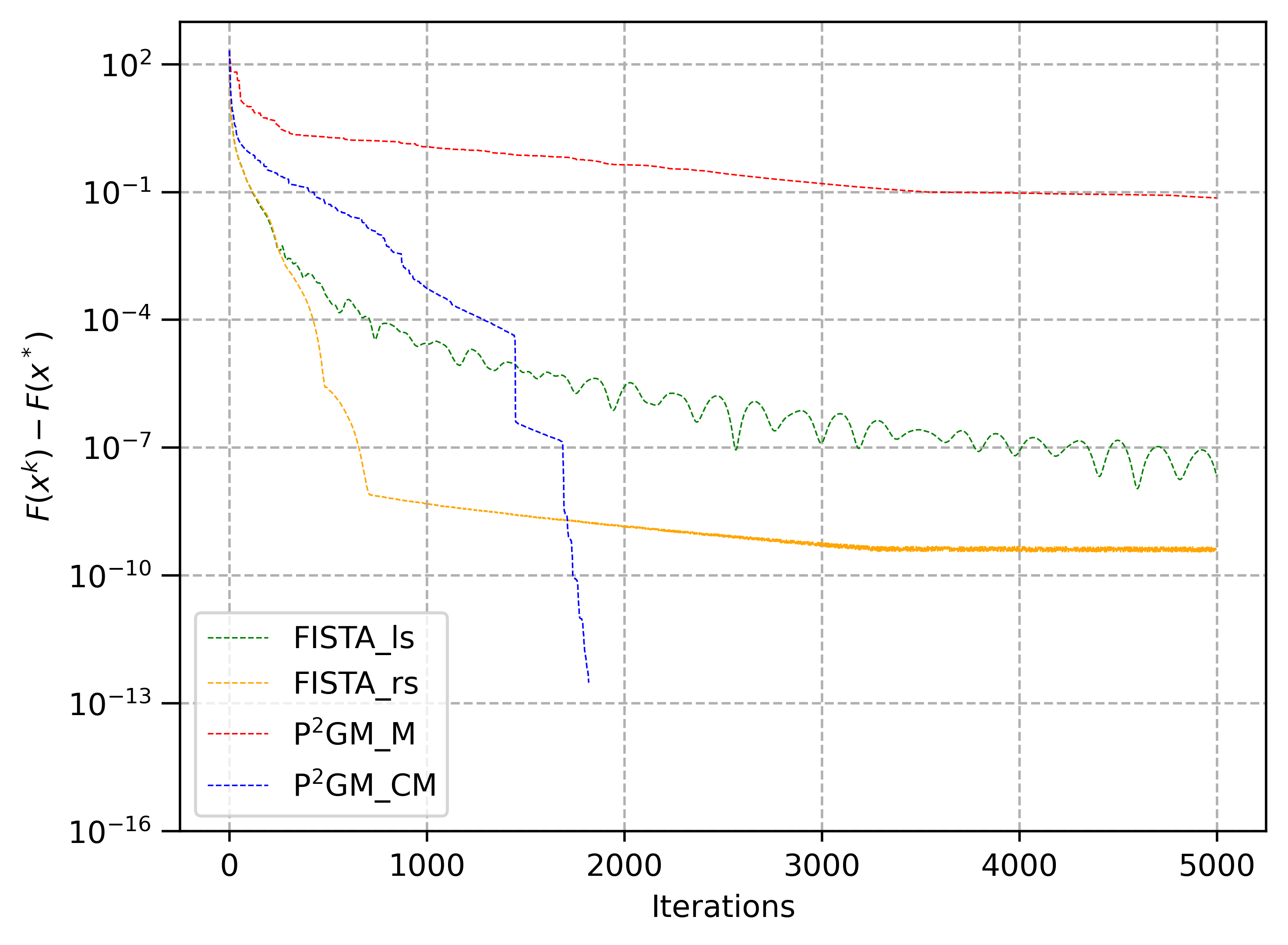} 
		\end{minipage}
	}
	\subfigure[CPU time]
	{
		\begin{minipage}[H]{.45\linewidth}
			\centering
			\includegraphics[scale=0.4]{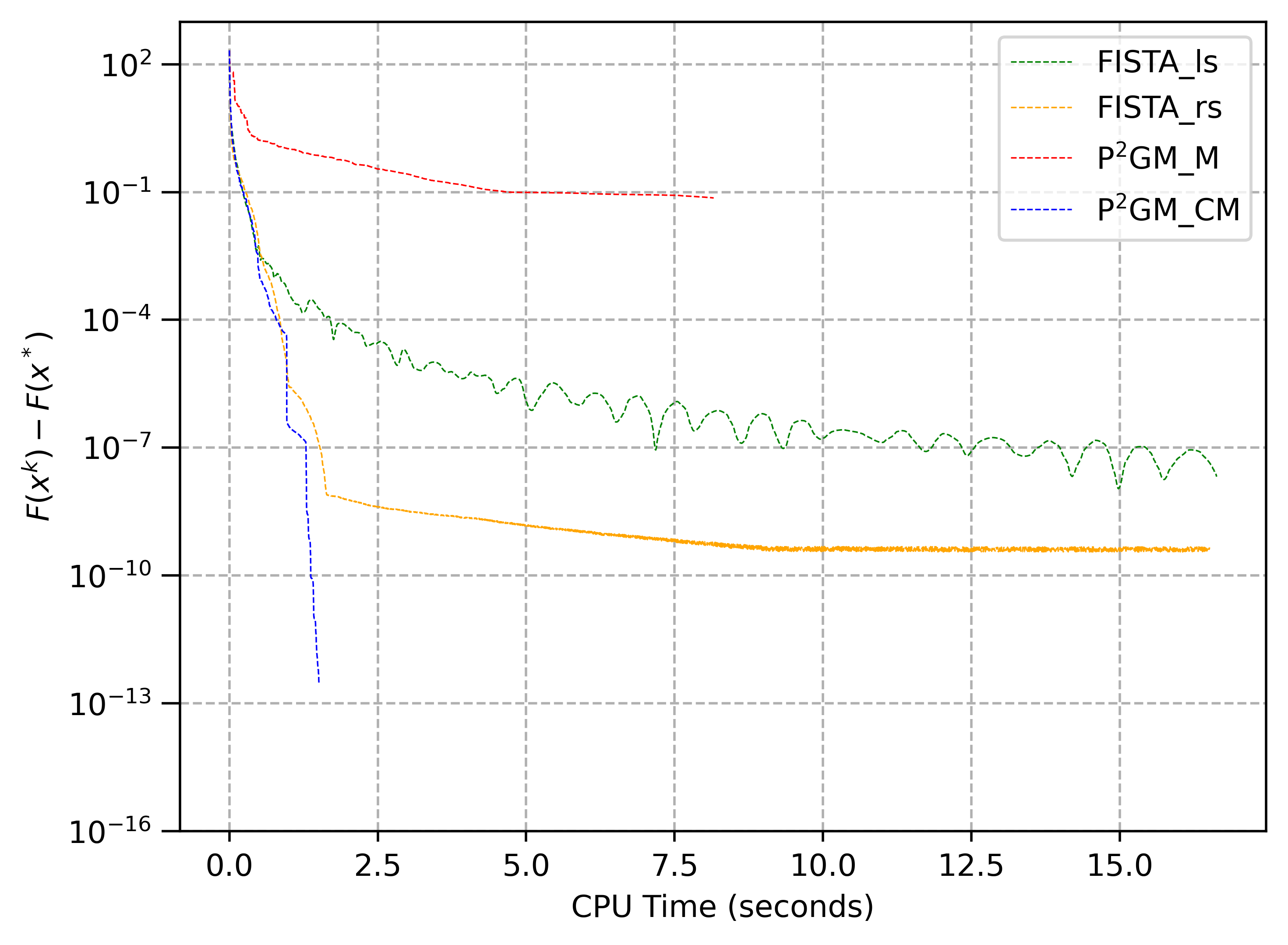} 
		\end{minipage}
	}
	\caption{Objective gaps w.r.t. iterations and CPU time for problem (\ref{qp}).}
	\label{f2}
\end{figure}

\subsection{Structured $\ell_{1}$ regularization quadratic problems}
Finally, we consider a structured composite optimization problem of the form
\begin{equation}\label{tv}
\min_{x\in\mathbb{R}^n} 
F(x) = f(x) + \lambda\|Ax\|_1,	
\end{equation}
where the smooth component is the quadratic function
\[
f(x) = \frac12 x^\top Q x + c^\top x .
\]
This problem combines an ill-conditioned quadratic objective with a structured linear operator in the nonsmooth term.

The problem parameters are generated as follows:

\begin{itemize}
	
	\item \textbf{Dimension:}
	$n = 100$, and the linear operator $A$ has $m=50$ rows.
	
	\item \textbf{Quadratic Term:}
	The matrix $Q\in\mathbb{S}_{++}^n$ is generated via spectral decomposition
	\[
	Q = U \Lambda U^\top,
	\]
	where $U$ is a random orthogonal matrix obtained from the QR factorization of a Gaussian matrix. The eigenvalues $\{\lambda_i\}$ are logarithmically spaced in the interval $[1,\kappa]$ with $\kappa = 5\times10^4$.
	
	\item \textbf{Linear Term:}
	The vector $c\in\mathbb{R}^n$ is sampled from a standard Gaussian distribution.
	
	\item \textbf{Linear Operator:}
	The matrix $A\in\mathbb{R}^{m\times n}$ is constructed via singular value decomposition
	\[
	A = U_A \Sigma V^\top,
	\]
	where $U_A$ and $V$ are random orthogonal matrices generated from Gaussian matrices via QR factorization. The singular values of $A$ are logarithmically spaced in $[1,\sigma_A]$ with $\sigma_A=\sqrt{5000}$.
	
	\item \textbf{Preconditioning Matrix:}
	To exploit the structure of $A$, we construct the preconditioning matrix
	\[
	P = V\left(\Sigma^\top\Sigma +
	\begin{bmatrix}
		\bm0_{m\times m}   & &  \\
		
		&  &\tilde{P}
	\end{bmatrix}\right)V^\top ,
	\]
where $\tilde P\in\mathbb{S}_{++}^{n-m}$ is chosen as the identity matrix. The inverse $P^{-1}$ can therefore be computed analytically from the block structure.
\end{itemize}

\begin{figure}[H]
	\centering
	\subfigure[Iterations]
	{
		\begin{minipage}[H]{.45\linewidth}
			\centering
			\includegraphics[scale=0.4]{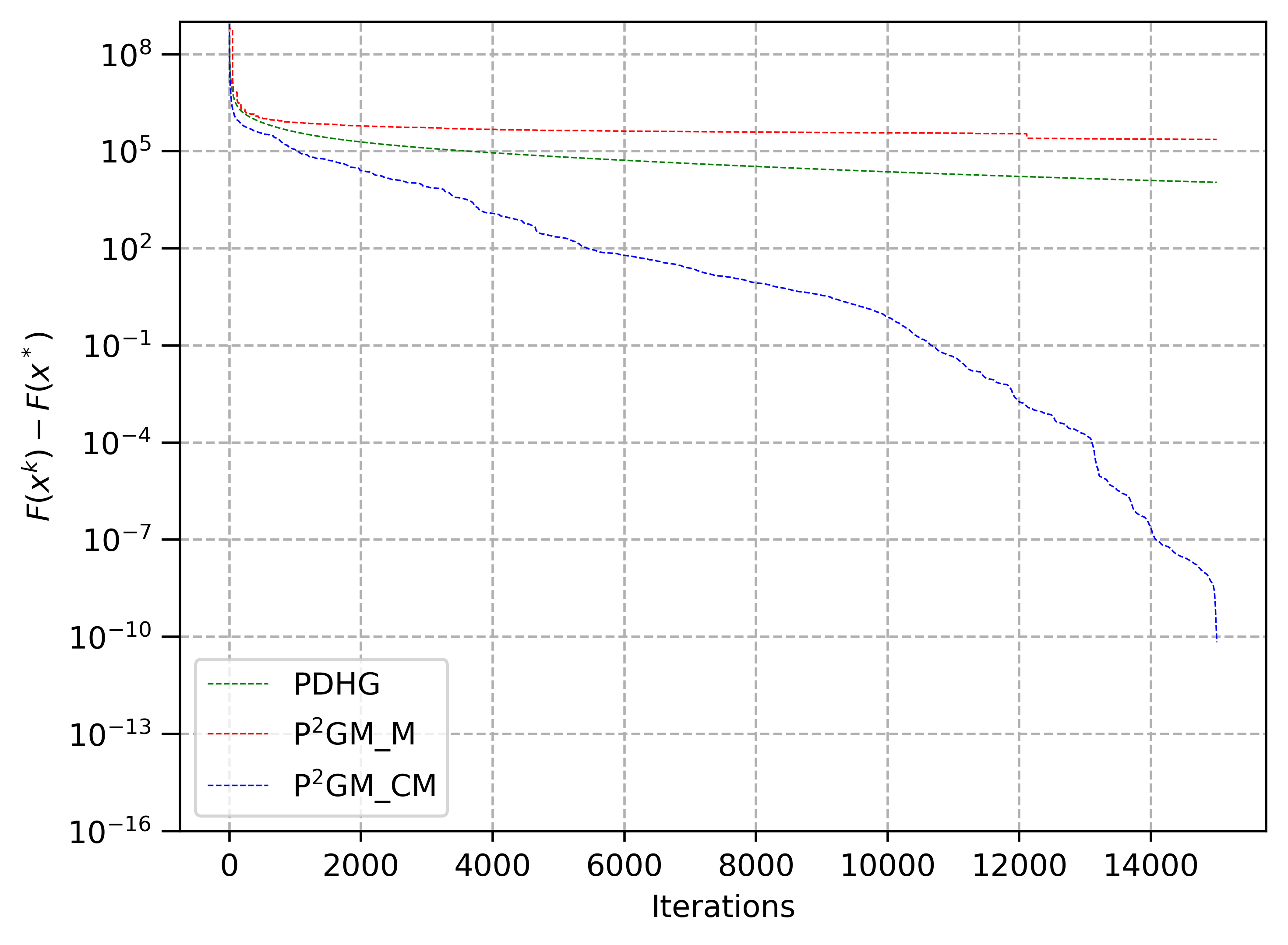} 
		\end{minipage}
	}
	\subfigure[CPU time]
	{
		\begin{minipage}[H]{.45\linewidth}
			\centering
			\includegraphics[scale=0.4]{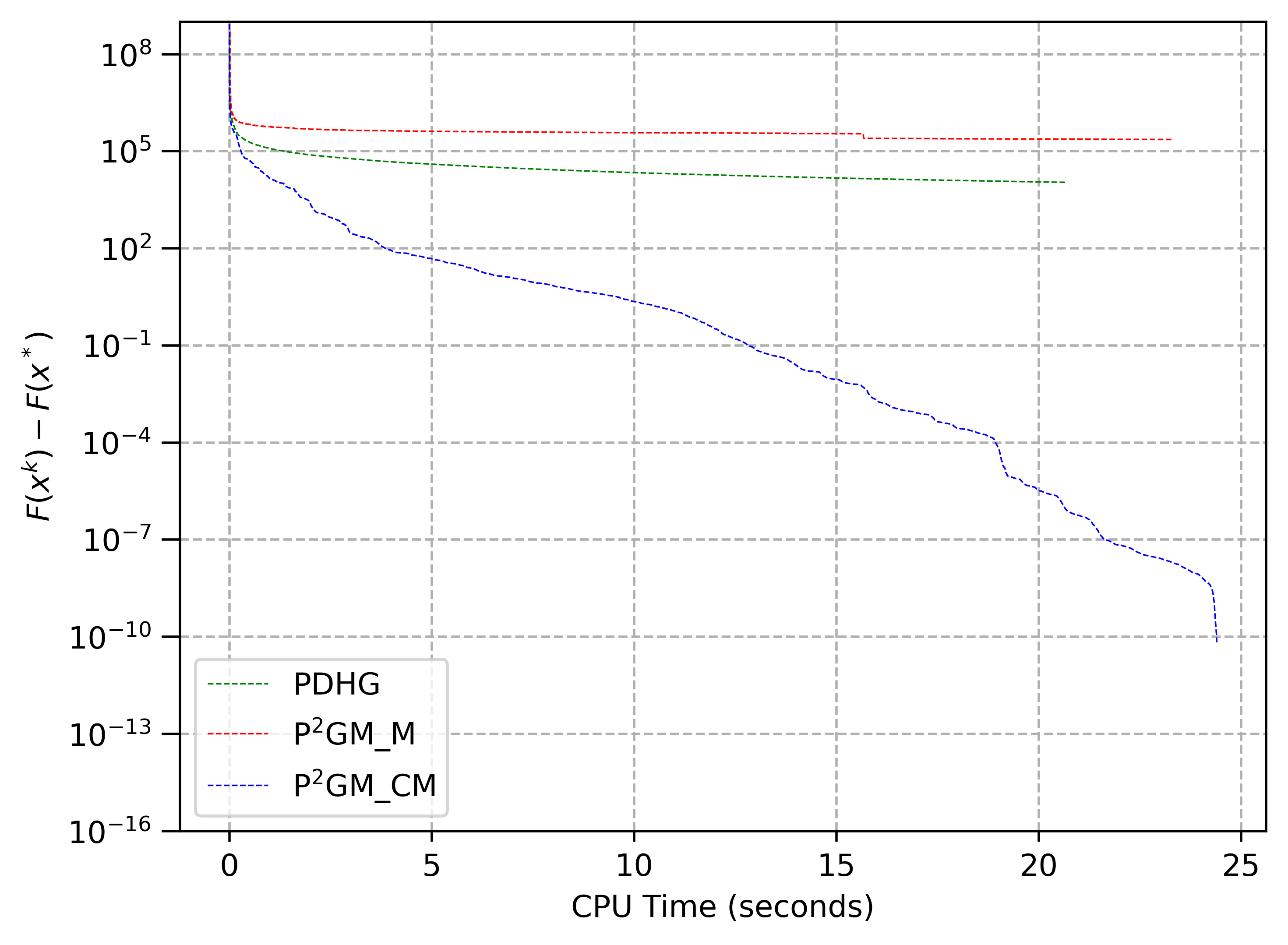} 
		\end{minipage}
	}
	\caption{Objective gaps w.r.t. iterations and CPU time for problem (\ref{tv}) with $\lambda=1/16$.}
	\label{f3}
\end{figure}
\subsection{Discussion of numerical results}
Figures \ref{f1}--\ref{f3} illustrate the convergence behavior of the tested algorithms on three representative problems: LASSO, simplex-constrained quadratic programming, and structured $\ell_{1}$ regularization problems. In all experiments, the proposed method $\mathrm{P^{2}GM}$\_CM consistently achieves the fastest decrease of the objective gap with respect to both iterations and CPU time.

A key observation from these figures is the clear performance gap between $\mathrm{P^{2}GM}$\_CM and its non-conjugate variant $\mathrm{P^{2}GM}$\_M. While both methods share the same preconditioning framework, the orthogonalized conjugate momentum used in $\mathrm{P^{2}GM}$\_CM significantly improves the efficiency of the subspace search, leading to much faster convergence. This suggests that incorporating curvature information through conjugate directions plays an essential role in accelerating the algorithm.

Moreover, compared with classical first-order methods such as FISTA and PDHG, the proposed approach demonstrates stronger robustness to ill-conditioning and maintains stable convergence across different problem structures. These results highlight the advantage of combining structure-exploiting preconditioning with conjugate momentum.
\section{Conclusions}\label{sec8}
In this paper we studied composite optimization problems involving linear operators and proposed a preconditioned proximal gradient framework with conjugate momentum from a subspace perspective. 
By exploiting the structure of the linear operator, we constructed a preconditioner that transforms the proximal gradient subproblem into a dual formulation whose solution admits closed-form expressions for several important classes of problems. 
This structure-driven preconditioning significantly simplifies the computation of proximal updates.

To further improve the convergence behavior, we introduced a subspace acceleration mechanism that incorporates curvature information through a proximal Newton model restricted to a low-dimensional subspace. 
By orthogonalizing the subspace basis with respect to the Hessian-induced inner product, the resulting two-dimensional nonsmooth problem can be efficiently approximated by a sequence of one-dimensional optimization problems. 
This design preserves computational efficiency while capturing useful second-order information.

We established global convergence of the proposed algorithm and proved a $Q$-linear convergence rate under standard smoothness and strong convexity assumptions. 
Numerical experiments on several representative applications, including LASSO problems, simplex-constrained quadratic programs, and structured $\ell_1$ regularization problems, demonstrated that the proposed method achieves competitive performance, especially on ill-conditioned problems.

Several directions for future research remain open. 
\begin{itemize}
	\item[$\bullet$] First, it would be of interest to extend the proposed framework to more general composite optimization problems in which the nonsmooth term may also be nonconvex. 
	Such an extension would broaden the applicability of the method to a wider class of modern machine learning and signal processing models.
	
	\item[$\bullet$] Second, inspired by the recent developments of the dimension-reduced second-order method \cite{ZGH2022}, it would be worthwhile to investigate the fast asymptotic convergence properties of the proposed algorithm. 
	In particular, understanding whether the subspace proximal framework can inherit superlinear or fast local convergence behavior remains an interesting theoretical question.
	
	\item[$\bullet$] Finally, our numerical experiments indicate that the orthogonalization of the momentum direction plays a crucial role in improving the practical performance of the algorithm. 
	This observation suggests a broader research direction: incorporating conjugate or orthogonalized momentum into other momentum-based optimization frameworks. 
	In particular, it would be interesting to explore whether similar ideas can be integrated into accelerated schemes such as Nesterov's method, potentially leading to new accelerated algorithms with improved robustness and convergence behavior.
\end{itemize}

\begin{acknowledgements}
 This work was funded by the National Key Research and Development Program of China [grant number 2023YFA1011504]; the Major Program of the National Natural Science Foundation of China [grant numbers 11991020, 11991024]; the Key Program of the National Natural Science Foundation of China [grant number 12431010]; the General Program of the National Natural Science Foundation of China [grant number 12171060]; NSFC-RGC (Hong Kong) Joint Research Program [grant number 12261160365]; the Team Project of Innovation Leading Talent in Chongqing [grant number CQYC20210309536]; the Natural Science Foundation of Chongqing [grant numbers ncamc2022-msxm01, CSTB2024NSCQ-LZX0140]; Major Project of Science and Technology Research Program of Chongqing Education Commission of China [grant number KJZD-M202300504]; the Chongqing Postdoctoral Research Project Special Grant [grant number 2024CQBSHTB1007];
 the Science and Technology Research Program of Chongqing Education Commission of China [grant number KJQN202400520] and Foundation of Chongqing Normal University [grant numbers 22XLB005, 22XLB006].

\end{acknowledgements}


\begin{thebibliography}{99}
\bibitem{ACFR2022}
H.~Attouch, Z.~Chbani, J.~Fadili, H.~Riahi.
\newblock First-order optimization algorithms via inertial systems with Hessian driven damping.
\newblock {Math. Program.}, 193:113--155, 2022.

\bibitem{BF2012}
S.~Becker, M.~J. Fadili.
\newblock A quasi-{N}ewton proximal splitting method.
\newblock In {Advances in Neural Information Processing Systems}, pages 2627--2635, 2012.

\bibitem{BT2009}
A.~Beck, M.~Teboulle.
\newblock A fast iterative shrinkage-thresholding algorithm for linear inverse problems.
\newblock {SIAM J. Imaging Sci.}, 2(1):183--202, 2009.

\bibitem{CP2011}
A.~Chambolle, T.~Pock.
\newblock A first-order primal-dual algorithm for convex problems with applications to imaging.
\newblock {J. Math. Imaging Vision}, 40:120--145, 2011.

\bibitem{CTY2025}
J.~Chen, L.~Tang, X.~M. Yang.
\newblock A subspace minimization Barzilai-Borwein method for multiobjective optimization problems.
\newblock {Comput. Optim. Appl.}, 92:155--178, 2025.

\bibitem{C2016}
L.~Condat.
\newblock Fast projection onto the simplex and the $\ell_1$ ball.
\newblock {Math. Program.}, 158:575--585, 2016.

\bibitem{DY1999}
Y.~Dai, Y.~Yuan.
\newblock A nonlinear conjugate gradient method with a strong global convergence property.
\newblock {SIAM J. Optim.}, 10(1):177--182, 1999.

\bibitem{FR1964}
R.~Fletcher, C.~Reeves.
\newblock Function minimization by conjugate gradients.
\newblock {Comput. J.}, 7(2):149--154, 1964.

\bibitem{HY2012}
B.~He, X.~Yuan.
\newblock Convergence analysis of primal-dual algorithms for a saddle-point problem: from contraction perspective.
\newblock {SIAM J. Imaging Sci.}, 5(1):119--149, 2012.

\bibitem{HS1952}
M.~R. Hestenes, E.~Stiefel.
\newblock Methods of conjugate gradients for solving linear systems.
\newblock {J. Res. Natl. Bur. Stand.}, 49(6):409--436, 1952.

\bibitem{LLLS2026}
M.~Lapucci, G.~Liuzzi, S.~Lucidi, M.~Sciandrone.
\newblock A globally convergent gradient method with momentum.
\newblock {Comput. Optim. Appl.}, 93:795--820, 2026.

\bibitem{LLL2026}
M.~Lapucci, G.~Liuzzi, S.~Lucidi, M.~Sciandrone, D.~Scuppa.
\newblock Projected gradient methods with momentum.
\newblock {arXiv preprint arXiv:2601.16683}, 2026.

\bibitem{LSS2014}
J.~Lee, Y.~Sun, M.~Saunders.
\newblock Proximal {N}ewton-type methods for minimizing composite functions.
\newblock {SIAM J. Optim.}, 24(3):1420--1443, 2014.

\bibitem{LC2022}
H.~Luo, L.~Chen.
\newblock From differential equation solvers to accelerated first-order methods for convex optimization.
\newblock {Math. Program.}, 195:735--781, 2022.

\bibitem{MP2018}
Y.~Malitsky, T.~Pock.
\newblock A first-order primal-dual algorithm with linesearch.
\newblock {SIAM J. Optim.}, 28(1):411--432, 2018.

\bibitem{N1983}
Y.~Nesterov.
\newblock A method for solving the convex programming problem with convergence rate $O(1/k^2)$.
\newblock {Soviet Math. Dokl.}, 27(2):372--376, 1983.

\bibitem{N2013}
Y.~Nesterov.
\newblock Gradient methods for minimizing composite objective function.
\newblock {Math. Program.}, 140:125--161, 2013.

\bibitem{OC2015}
B.~O'Donoghue, E.~Cand\`es.
\newblock Adaptive restart for accelerated gradient schemes.
\newblock {Found. Comput. Math.}, 15:715--732, 2015.

\bibitem{PDB2020}
Y.~Park, S.~Dhar, S.~Boyd, M.~Shah.
\newblock Variable metric proximal gradient method with diagonal {B}arzilai-{B}orwein stepsize.
\newblock In {ICASSP 2020 IEEE International Conference on Acoustics, Speech and Signal Processing}, 3597--3601, 2020.

\bibitem{P1964}
B.~T. Polyak.
\newblock Some methods of speeding up the convergence of iteration methods.
\newblock {USSR Comput. Math. Math. Phys.}, 4(5):1--17, 1964.

\bibitem{PR1969}
E.~Polak, G.~Ribi\`ere.
\newblock Note sur la convergence de m\'ethodes de directions conjugu\'ees.
\newblock {Rev. Fran\c{c}aise Inform. Rech. Op\'er.}, 3:35--43, 1969.

\bibitem{QGHY2025}
Z.~Qu, W.~Gao, O.~Hinder, Y.~Ye, Z.~Zhou.
\newblock Optimal diagonal preconditioning.
\newblock {Oper. Res.}, 73(3):1479--1495, 2025.

\bibitem{SGB2014}
K.~Scheinberg, D.~Goldfarb, X.~Bai.
\newblock Fast first-order methods for composite convex optimization with backtracking.
\newblock {Found. Comput. Math.}, 14:389--417, 2014.

\bibitem{ST2016}
K.~Scheinberg, X.~Tang.
\newblock Practical inexact proximal quasi-{N}ewton method with global complexity analysis.
\newblock {Math. Program.}, 160(1):495--529, 2016.

\bibitem{SDSJ2022}
B.~Shi, S.~ Du, W.~Su, M.~Jordan.
\newblock Understanding the acceleration phenomenon via high-resolution differential equations.
\newblock {Math. Program.}, 195:79--148, 2022.

\bibitem{SBC2016}
W.~Su, S.~Boyd, E.~Cand\`es.
\newblock A differential equation for modeling {N}esterov's accelerated gradient method: theory and insights.
\newblock {J. Mach. Learn. Res.}, 17(153):1--43, 2016.

\bibitem{YS1995}
Y.~ Yuan, J.~Stoer.
\newblock A subspace study on conjugate gradient algorithms.
\newblock {Z. Angew. Math. Mech.}, 75(1):69--77, 1995.

\bibitem{ZGH2022}
C.~Zhang, D.~Ge, C.~He, B.~Jiang, Y.~Jiang, Y.~Ye.
\newblock DRSOM: A dimension reduced second-order method.
\newblock {arXiv preprint arXiv:2208.00208}, 2022.
\end{thebibliography}
\end{document}